\documentclass[12pt]{amsart}
\usepackage{amsfonts}
\usepackage{mathrsfs}
\usepackage{amscd,amsmath,amssymb,amsfonts}
\usepackage{color}
\usepackage[all]{xy}
\theoremstyle{plain}
\newtheorem{thm}{Theorem}
\newtheorem{lem}[thm]{Lemma}
\newtheorem{cor}[thm]{Corollary}
\newtheorem{prop}[thm]{Proposition}

\theoremstyle{definition}
\newtheorem{defn}[thm]{Definition}

\newtheorem{question}[thm]{Question}

\newtheorem{rmks}[thm]{Remarks}

\newtheorem{nota}[thm]{Notation}

\numberwithin{thm}{section} \numberwithin{equation}{section}

\newcommand{\ga}[2]{\begin{gather}\label{#1}#2 \end{gather}}

\newcommand{\sC}{{\mathcal C}}
\newcommand{\sD}{{\mathcal D}}
\newcommand{\sE}{{\mathcal E}}
\newcommand{\sF}{{\mathcal F}}

\newcommand{\sH}{{\mathcal H}}

\newcommand{\sL}{{\mathcal L}}
\newcommand{\sM}{{\mathcal M}}
\newcommand{\sN}{{\mathcal N}}
\newcommand{\sO}{{\mathcal O}}
\newcommand{\sP}{{\mathcal P}}
\newcommand{\sQ}{{\mathcal Q}}
\newcommand{\sR}{{\mathcal R}}

\newcommand{\sU}{{\mathcal U}}

\newcommand{\sW}{{\mathcal W}}
\newcommand{\sX}{{\mathcal X}}

\def\wt#1{\widetilde {#1}}

\begin{document}

\title{F-regular type of Moduli spaces and Verlinde Formula}
\author{Xiaotao Sun}
\address{Center of Applied Mathematics, School of Mathematics, Tianjin University, No.92 Weijin Road, Tianjin 300072, P. R. China}
\email{xiaotaosun@tju.edu.cn}
\author{Mingshuo Zhou}
\address{Center of Applied Mathematics, School of Mathematics, Tianjin University, No.92 Weijin Road, Tianjin 300072, P. R. China}
\email{zhoumingshuo@amss.ac.cn}
\date{February 18, 2018}
\thanks{}
\begin{abstract} We prove that moduli spaces of semistable parabolic bundles and generalized parabolic sheaves (GPS)
with a fixed determinant on a smooth projective curve are globally $F$-regular type. As an application, we prove vanishing theorems
on the moduli spaces of semistable parabolic sheaves on a singular curve, which combining with Factorization theorems in \cite{Su1} and \cite{Su2}
give two recurrence relations among dimensions of spaces of generalized theta functions. By using of these recurrence relations, we prove an
explicit formula (Verlinde formula) for the dimension of spaces of generalized theta functions.
\end{abstract}
\keywords{Globally $F$-regular type,  Moduli spaces, Parabolic sheaves}
\subjclass{Algebraic Geometry, 14H60, 14D20}
\maketitle
\begin{quote}

\end{quote}
\section{Introduction}
 Let $X$ be a variety over a perfect field $k$ of characteristic $p>0$ and $F:X\to X$ be the Frobenius morphism. The $X$ is called
  $F$-split (Frobenius split) if the natural homomorphism $\sO_X\hookrightarrow F_*\sO_X$ is split. Although most of
 projective varieties are not $F$-split, some important varieties are $F$-split. For example,  flag varieties and
 their Schubert subvarieties (cf. \cite{MeRa}, \cite{RR}), the product of
two flag varieties for the same group $G$ (cf. \cite{MeRa88}) and cotangent bundles of flag varieties (cf. \cite{KLT}) are proved to be $F$-split.
An example, which is more closer to this article, should be mentioned. Mehta-Ramadas proved in \cite{MeR} that for a generic
nonsingular projective curve $C$ of genus $g$ over an algebraically closed field of characteristic $p\ge 5$, the moduli space of semistable parabolic bundles of rank $2$ on $C$ is $F$-split, and made conjecture that for any nonsingular projective curve $C$, the moduli space of semistable parabolic bundles of rank $2$ on $C$ with a fixed determinant is $F$-split.

The notion of globally $F$-regular variety was introduced by K. E. Smith in \cite{Sm}, a variety $X$ is called globally $F$-regular if for any effective
divisor $D$, the natural homomorphism $\sO_X\hookrightarrow F^e_*\sO_X(D)$ is split for some integer $e>0$. It is clear that globally $F$-regular varieties must be
$F$-split. Also, some well-known $F$-split varieties include toric varieties and Schubert varieties are proved in \cite{Sm} and \cite{LRT} to be globally $F$-regular.
Thus it is natural to extend Mehta-Ramadas conjecture: the moduli spaces $\sU_{C,\,\omega}^L$ of semistable parabolic bundles of rank $r$ with a fixed determinant $L$ on any smooth curves $C$ (parabolic structures determined by a given data $\omega$) are globally $F$-regular varieties.
It remains to be a very difficult open problem, we will study its characteristic zero analogy in this article.

A variety $X$ over a field of characteristic zero is called globally $F$-regular type (resp. $F$-split type) if its modulo $p$ reduction $X_p$ is
globally $F$-regular (resp. $F$-split) for a dense set of $p$. Projective varieties $X$, which are globally F-regular type, have remarkable geometric and
cohomological properties: (1) $X$ must be normal, Cohen-Macaulay with rational singularities, and must have log terminal singularities if it is $\mathbb{Q}$-Gorenstein;
(2) $H^i(X,\sL)=0$ for $i>0$ and nef line bundle $\sL$.

Let $\sU_{C,\,\omega}$ be moduli spaces of semistable parabolic bundles
of rank $r$ and degree $d$ on smooth curves $C$ of genus $g\ge 0$ with parabolic structures determined by $\omega=(k,\{\vec n(x),\vec a(x)\}_{x\in I})$, and
$${\rm det}: \sU_{C,\,\omega}\to J^d_C$$
be the determinant morphism. For any $L\in J^d_C$, the fiber $$\sU_{C,\,\omega}^L:={\rm det}^{-1}(L)$$ is called moduli space of semistable parabolic bundles with
a fixed determinant $L$. Then the first main result in this article is

\begin{thm}[See Theorem \ref{thm3.7} and Theorem \ref{thm4.1}]\label{thm1.1} The moduli spaces $\sU_{C,\,\omega}^L$ are of globally $F$-regular type and
$${\rm H}^i(\sU_{C,\,\omega},\sL)=0 \quad \forall\,\,i>0$$
for any ample line bundle $\sL$ on $\sU_{C,\,\omega}$.
\end{thm}

When the projective curve $C$ has exactly one node (irreducible, or reducible), the moduli space $\sU_{C,\,\omega}$ is not normal and its normalization
is a moduli space $\sP_{\omega}$ of semistable generalized parabolic sheaves (GPS) on $\wt C$ (where $\wt{C}$ is normalization of $C$). There exist a similar determinant morphism
${\rm det}:\sP_{\omega}\to J^d_{\wt{C}}$. For any $L\in J^d_{\wt{C}}$, the fiber
$$\sP^L_{\omega}:={\rm det}^{-1}(L)$$
is called a moduli space of semistable generalized parabolic sheaves (GPS) with a fixed determinant $L$ on $\wt C$. Then the second main result in this article is
\begin{thm}[See Theorem \ref{thm3.14} and Theorem \ref{thm3.22}]\label{thm1.2} The moduli spaces $\sP_{\omega}^L$ are of globally $F$-regular type.
\end{thm}

As an application of this theorem, we prove vanishing theorems on $\sU_{C,\,\omega}$ for singular curves $C$.

\begin{thm} [See Theorem \ref{thm4.3} and Theorem \ref{thm4.4}]\label{thm1.3} When $C$ is irreducible with at most one node, we have
$$H^1(\sU_{C,\,\omega},\Theta_{\sU_{C,\,\omega}})=0$$
where $\Theta_{\sU_{C,\,\omega}}$ is the theta line bundle. When $C$ is reducible with at most one node, then for any ample line bundle $\sL$ on $\sU_{C,\,\omega}$
$$H^1(\sU_{C,\,\omega},\sL)=0.$$
\end{thm}

To describe the idea of proof, recall that the moduli space $\sU^L_{C,\,\omega}$ is
a GIT quotient $(\sR^{ss}_{\omega})^L//{\rm SL}(V)$, where $(\sR^{ss}_{\omega})^L\subset\sR^L_F$ is a open set of a quasi-projective variety $\sR^L_F$ ( i.e. the
set of GIT semistable points respect to a polarization $\Theta_{\sR,\,\omega}$ determined by $\omega$). Then our idea is to find
a flag bundle ${\sR'}^L_F\xrightarrow{\hat f}\sR^L_F$ over $\sR^L_F$ and a data $\omega'$ such that $\sU^L_{C,\,\omega'}=({\sR'}_{\omega'}^{ss})^L//{\rm SL}(V)$
is a Fano variety with an open subvariety $X\subset\sU^L_{C,\,\omega'}$ and a morphism $X\xrightarrow{f}\sU^L_{C,\,\omega}$ satisfying
$f_*\sO_X=\sO_{\sU^L_{C,\,\omega}}.$
Since Fano varieties are globally $F$-regular type by Proposition 6.3 of \cite{Sm}, so are $X$ and $\sU^L_{C,\,\omega}$ if the equality $f_*\sO_X=\sO_{\sU^L_{C,\,\omega}}$
commutes with modulo $p$ reductions for a dense set of $p$. To prove that $f_*\sO_X=\sO_{\sU^L_{C,\,\omega}}$
commutes with modulo $p$ reductions for a dense set of $p$, one has to show in particular that a GIT quotient over $\mathbb{Z}$ must commute
with modulo $p$ reductions for a dense set of $p$, which is Lemma \ref{lem2.9} (we thought at first that Lemma \ref{lem2.9} must be well-known to experts,
but finally we are not able to find any reference). We have formulated our idea in Proposition \ref{prop2.10} for a general setting, the proof of Theorem \ref{thm1.1} and Theorem \ref{thm1.2} becomes to check
conditions in Proposition \ref{prop2.10}. To prove Theorem \ref{thm1.3}, consider the normalization morphism $\phi:\sP_{\omega}\to\sU_{C,\,\omega}$, then we believe that
\ga{1.1} {\phi^*: H^i(\sU_{C,\,\omega},\sL)\to H^i(\sP_{\omega}, \phi^*\sL)}
is injective. Unfortunately, we are only able to check injectivity of \eqref{1.1} when $i=1$. Thus Theorem \ref{thm1.3} is reduced to
the proof of some vanishing theorems on $\sP_{\omega}$ (cf. Theorem \ref{thm4.3}, Theorem \ref{thm4.4}).

Our starting motivation was to give an algebraic geometric proof of a formula (so called Verlinde formula) from Rational Conformal Field Theories (RCFT)
(cf. \cite{Ver}). RCFT is defined to be a functor which associates
to any marked projective curve $(C, I, \{\vec a(x)\}_{x\in I})$ a finite-dimensional vector
space $V_C(I, \{\vec a(x)\}_{x\in I})$ satisfying certain axioms (A0--A4) (cf. \cite{Be3} for the detail). The axioms, in particular the factorization rules (A2 and A4),
can be encoded in a finite-dimensional $\mathbb{Z}$-algebra, the fusion ring of the theory. An explicit formula (so called Verlinde formula) for the dimension of $V_C(I, \{\vec a(x)\}_{x\in I})$ can be obtained in terms of the characters of the fusion ring (See Proposition 3.3 of \cite{Be3}).

An important example of RCFT was constructed for a Lie algebra $\mathfrak{g}$ in \cite{TUY} (WZW-models) by associating a space $V_C(\mathfrak{g}, I, \{\vec a(x)\}_{x\in I})$ of conformal blocks to a marked projective curve $(C, I, \{\vec a(x)\}_{x\in I})$. It is this example that relates RCFT to algebraic geometry
when the space of conformal blocks was proved to be the spaces of generalized theta functions on moduli spaces of parabolic $G$-bundles (
${\rm Lie}(G)=\mathfrak{g}$) (cf. \cite{BL}, \cite{Pa} and \cite{Fa}). Then the characters of its fusion ring are determined in terms of representations
of $\mathfrak{g}$  (cf. \cite{Be2} for $\mathfrak{g}=\mathfrak{sl}(r),\,\mathfrak{sp}(r)$ and \cite{Fa} for all classical algebras). Thus an explicit formula
(Verlinde formula) for the dimension of spaces of generalized theta functions is proved. This kind of proof was called \textit{infinite dimensional proof} in \cite{Be}.

It is natural to ask if associating ${\rm H}^0(\sU_{C,\,\omega},\Theta_{\sU_{C,\,\omega}})$ to a marked projective curve $(C, \omega)$,
where $\omega=(k,\,\{\vec n(x),\vec a(x)\}_{x\in I})$, satisfies the axioms of RCFT ? At least, one can ask if the numbers
$$D_g(r,d,\omega)=dim {\rm H}^0(\sU_{C,\,\omega},\Theta_{\sU_{C,\,\omega}})$$
satisfy the factorization rules ? so that an explicit formula of $D_g(r,d,\omega)$ can be proved without using conformal blocks. The idea is to consider a family
$\{C_t\}_{t\in \Delta}$ of curves degenerating to a curve $C_0$ with exactly one node. A factorization theorem of
${\rm H}^0(\sU_{C_0,\,\omega_0},\Theta_{\sU_{C_0,\,\omega_0}})$ for irreducible $C_0$ was proved in \cite{Su1} (see also \cite{NR} for $r=2$), and for reducible $C_0$ it
was proved in \cite{Su2}. To end the story, one has to show that $dim {\rm H}^0(\sU_{C_t,\,\omega_t},\Theta_{\sU_{C_t,\,\omega_t}})$ is independent of $t\in\Delta$,
which follows clearly that ${\rm H}^1(\sU_{C_t,\,\omega_t},\Theta_{\sU_{C_t,\,\omega_t}})=0$. When $C_0$ is irreducible, this kind of vanishing theorems
were proved in \cite{Su1} (see also \cite{NR} for $r=2$) under assumption that $g\ge 3$. Although we have shown in \cite{Su3} that $dim {\rm H}^0(\sU_{C_t,\,\omega_t},\Theta_{\sU_{C_t,\,\omega_t}})$ is constant for $t\neq 0$ without using vanishing theorem (cf. Corollary 4.8 of \cite{Su3}), the vanishing theorems
for singular curves $C_0$ are needed in order to show that
$$dim {\rm H}^0(\sU_{C_t,\,\omega_t},\Theta_{\sU_{C_t,\,\omega_t}})=dim {\rm H}^0(\sU_{C_0,\,\omega_0},\Theta_{\sU_{C_0,\,\omega_0}})\quad (\forall\,\,t\in\Delta).$$
Then, by Theorem \ref{thm1.3}, we have the recurrence relations

\begin{cor}[See Theorem \ref{thm4.6} and Theorem \ref{thm4.12}]\label{cor1.4} For partitions $g=g_1+g_2$ and $I=I_1\cup I_2$, let
$$W_k=\{\,\lambda=(\lambda_1,...,\lambda_r)\,|\, 0=\lambda_r\le\lambda_{r-1}\le\cdots\le\lambda_1\le k\,\}$$
$$W'_k=\left\{\,\lambda\in W_k\,\,\mid\,\,\left(\sum_{x\in
I_1}\sum^{l_x}_{i=1}d_i(x)r_i(x)+\sum^r_{i=1}\lambda_i\right) \equiv 0({\rm mod}\,\,r)\right\}.$$
Then we have the following recurrence relation
\ga{1.2} {D_g(r,d,\omega)=\sum_{\lambda\in W'_k} D_{g_1}(r, 0,\omega_1^{\lambda})\cdot D_{g_2}(r,d,\omega_2^{\lambda}),}
\ga{1.3}{D_g(r,d,\omega)=\sum_{\mu}D_{g-1}(r,d,\omega^{\mu})}
where $\mu=(\mu_1,\cdots,\mu_r)$ runs through $0\le\mu_r\le\cdots\le
\mu_1< k$ and $\omega^{\mu}$, $\omega_1^{\lambda}$, $\omega_2^{\lambda}$ are explicitly determined by $\mu$ and $\lambda$.
\end{cor}

We describe briefly content of the article. In Section 2, we collect notions and properties of globally $F$-regular type varieties, in particular,
we formulate and prove Proposition \ref{prop2.10}, which is our technical tool to show globally $F$-regular type of GIT quotients. In Section 3, we prove Theorem \ref{thm1.1}
and Theorem \ref{thm1.2}.  In Section 4, we prove Theorem \ref{thm1.3} and Corollary \ref{cor1.4} where recurrence relation \eqref{1.3} follows Theorem \ref{thm1.3} and
the factorization theorem in \cite{Su1}. But recurrence relation \eqref{1.2} is obtained by using factorization theorem in \cite{Su2} and Hecke transformation.
In Section 5, by using of recurrence relations \eqref{1.2} and \eqref{1.3}, we are able to check Verlinde formula (except $g=0$, $d=0$ and $|I|=3$):
$$\aligned &D_g(r,d,\omega)=(-1)^{d(r-1)}\left(\frac{k}{r}\right)^g(r(r+k)^{r-1})^{g-1}\\&\sum_{\vec v}
\frac{{\rm exp}\left(2\pi i\left(\frac{d}{r}-\frac{|\omega|}{r(r+k)}\right)\sum_{i=1}^{r}v_i\right)S_{\omega}\left({\rm exp}\,2\pi i\frac{\vec v}{r+k}\right)} {\prod_{i<j}\left(2\sin\,\pi \frac{v_i-v_j}{r+k}\right)^{2(g-1)}}\endaligned$$
where $\vec v=(v_1,v_2,\ldots,v_r)$ runs through the integers $$0=v_r<\cdots <v_2<v_1<r+k.$$\\[.2cm]
{\it Acknowledegements:} We would like to thank C. S. Seshadri for a number of emails of discussions about our Lemma \ref{lem2.9}, and
T. R. Ramadas for suggesting some helpful references. The discussions of globally $F$-regular type varieties with
K. Schwede and K. E. Smith (by emails) are also helpful, we thank them very much.

\section{Globally F-regular varieties}

The main result in this section is Proposition \ref{prop2.10}, which will be a key technical tool for us to prove globally F-regular type of GIT quotients. We first collect some notions and facts of globally F-regular varieties over a perfect field $k$ of positive characteristic and recall the definition of
globally F-regular type of varieties over a field of characteristic zero. Our main references here are \cite{Br}, \cite{Sc} and \cite{Sm}.

Let $X$ be a variety over a perfect field $k$ of $char (k)=p>0$, $$F:X\to X$$
be the Frobenius map and $F^e:X\to X$
be the e-th iterate of Frobenius map. When $X$ is normal, for any (weil) divisor $D\in Div(X)$,
$$\sO_X(D)(V)=\{\,f\in K(X)\,|\, div_V(f)+D|_V\ge 0\,\},\quad \forall\,\,V\subset X$$
is a reflexive subsheaf of constant sheaf $K=K(X)$. In fact, we have
$$\sO_X(D)=j_*\sO_{X^{sm.}}(D)$$
where $j:X^{sm.}\hookrightarrow X$ is the open set of smooth points, and $\sO_X(D)$ is an invertible sheaf
if and only if $D$ is a Cartier divisor.

\begin{defn}\label{defn2.1} A normal variety $X$ over a perfect field is called
\emph{stably Frobenius $D$-split} if $\sO_X\to F^e_*\sO_X(D)$ is split for some $e>0$.
$X$ is called \emph{globally F-regular} if $X$ is stably Frobenius $D$-split
for any effective divisor $D$.
\end{defn}

The advantage of this definition is that any open set $U\subset X$ of a globally F-regular variety $X$ is globally F-regular. Its disadvantage
is the requirement of normality of $X$. When $X$ is not normal, one possible remedy of Definition \ref{defn2.1} is to require that $D$ is a Cartier divisor. Then it loses the
advantage that any open set $U\subset X$ is globally F-regular since a Cartier divisor on $U$
may not be extended to a Cartier divisor on $X$. On the other hand, when $X$ is a projective variety and is stably Frobenius $D$-split for any effective Cartier $D$, then $X$ must be
normal and Cohen-Macaulay according to K. E. Smith (Theorem 3.10 and Theorem 4.1 of \cite{Sm}).
When $X$ is a projective variety, we recall the following proposition (due to K. E. Smith).

\begin{prop}[Theorem 3.10 of \cite{Sm}]\label{prop2.2} Let $X$ be a projective variety over a perfect field. Then the following statements are equivalent.
\begin{itemize} \item [(1)] $X$ is normal and is stably Frobenius $D$-split for any effective $D$;
\item [(2)] $X$ is stably Frobenius $D$-split for any effective Cartier $D$;
\item [(3)] For any ample line bundle $\sL$, the section ring of $X$
$$R(X,\sL)=\bigoplus_{n=0}^{\infty}H^0(X,\sL^n)$$
is strongly F-regular.
\end{itemize}
\end{prop}

\begin{proof} It is clear that $(1)\Rightarrow (2)$, and $(2)\Rightarrow (3)$ is proved in Theorem 3.10 of \cite{Sm}. That $(3)\Rightarrow (1)$ is a modification of the proof
in \cite{Sm}. By Theorem 4.1 of \cite{Sm}, $X$ is normal and Cohen-Macaulay.  Let $X^{sm.}\subset X$ be the open set of smooth points, then $R(X,\sL)=R(X^{sm.},\sL)$ and, for any effective $D\in Div(X)$,
$D\cap X^{sm.}$ is an effective Cartier  divisor on $X^{sm.}$. Then the proof of $(1)\Rightarrow (3)$ in Theorem 3.10 of \cite{Sm} implies
that $X^{sm.}$ is stably Frobenius $D\cap X^{sm.}$-split, which implies that $X$ is stably Frobenius $D$-split.
\end{proof}

A variety $X$ is called \emph{Frobenius split} if $\sO_X\to F_*\sO_X$ is split. In particular,
any \emph{globally F-regular} variety is Frobenius split. Let $X\xrightarrow{f}Y$ be a morphism and $f_*\sO_X=\sO_Y$, then any splitting map
$F_*\sO_X\xrightarrow{\psi}\sO_X$ of $\sO_X\to F_*\sO_X$ induces a splitting map $F_*\sO_Y=F_*f_*\sO_X=f_*F_*\sO_X\xrightarrow{f_*\psi} f_*\sO_X=\sO_Y$. There is a generalization
of above useful observation.

\begin{lem}[Corollary 6.4 of \cite{Sc}]\label{lem2.3} Let $f:X\to Y$ be a morphism of varieties over a perfect field $k$ of $char(k)=p>0$.
If the natural map $\sO_Y\xrightarrow{i} f_*\sO_X$ splits and $X$ is globally F-regular, then $Y$ is stably Frobenius $D$-split for any effective Cartier divisor $D$, and
it is globally F-regular when $Y$ is normal.
\end{lem}

\begin{proof} For any Cartier divisor $D\in Div(Y)$ defined by a section $s\in\Gamma(Y,\sO_Y(D))$, let $H=f^*D$ and $F_*^e\sO_X(H)\xrightarrow{h}\sO_X$ be a splitting of
$\sO_X\to F^e_*\sO_X\xrightarrow{F^e_*f^*(s)} F_*^e\sO_X(H)$, and
 $f_*\sO_X\xrightarrow{j}\sO_Y$ be a splitting of $\sO_Y\xrightarrow{i} f_*\sO_X$.
Then $\sO_Y(D)\xrightarrow{1\otimes i}\sO_Y(D)\otimes f_*\sO_X=f_*\sO_X(H)$ induces
$$F^e_*\sO_Y(D)\xrightarrow{F^e_*1\otimes i}F^e_*f_*\sO_X(H)=f_*F^e_*\sO_X(H)\xrightarrow{f_*h}f_*\sO_X\xrightarrow{j}\sO_Y$$
is a splitting of $\sO_Y\to F^e_*\sO_Y\xrightarrow{F^e_*s} F^e_*\sO_Y(D)$. When $Y$ is normal, let $Y_0\subset Y$ be the open set of smooth points, $Y$ is globally F-regular if and only if $Y_0$ is stably
Frobenius $D$-split for any effective Cartier divisor $D\in Div(Y_0)$, which is true by applying above argument to $f^{-1}(Y_0)\xrightarrow{f} Y_0$.
\end{proof}

For any scheme $X$ of finite type over a field $K$ of
characteristic zero, there is a
finitely generated $\mathbb{Z}$-algebra $A\subset K$ and an $A$-flat
scheme $$X_A\to S={\rm Spec}(A)$$ such that $X_K=X_A\times_S{\rm
Spec}(K)\cong X$. $X_A\to S={\rm Spec}(A)$ is called an integral model of $X/K$, and a closed fiber
$X_s=X_A\times_S{\rm Spec}(\overline{k(s)})$ is
called "\textbf{modulo $p$ reduction of $X$}" where $p={\rm
char}(k(s))>0$.

\begin{defn}\label{defn2.4} A variety $X$ over a field of characteristic zero is said to be of \emph{globally F-regular type} (resp.\emph{ F-split type}) if its "\textbf{modulo $p$ reduction of $X$}"  are globally F-regular (resp. \emph{F-split })
for a dense set of $p$.
\end{defn}

An equivalent definition of \emph{globally F-regular type} for a projective variety $X$ is that its modulo $p$ reductions (for a dense set of $p$) are stably Frobenius $D$-split along any effective Cartier divisor $D$, which do not require normality of its modulo $p$ reductions prior to the definition.
Projective varieties of \emph{globally F-regular type} have many nice properties and a good vanishing theorem of cohomology.

\begin{thm}[Corollary 5.3 and Corollary 5.5 of \cite{Sm}]\label{thm2.5} Let $X$ be a projective variety over a field of characteristic zero. If $X$ is of globally F-regular type,
then we have \begin{itemize} \item [(1)] $X$ is normal, Cohen-Macaulay with rational singularities. If $X$ is $\mathbb{Q}$-Gorenstein, then $X$ has log terminal singularities.
\item [(2)] For any nef line bundle $\sL$ on $X$, we have $H^i(X,\sL)=0$ when $i>0$. In particular, $H^i(X,\sO_X)=0$ whenever $i>0$.
\end{itemize}
\end{thm}

A normal projective variety $X$ is called a \emph{Fano variety} if
$$\omega_X^{-1}=\sH om_{\sO_X}(\omega_X,\sO_X)$$ is an ample line bundle. One of important examples of \emph{globally F-regular type} varieties is

\begin{prop}\label{prop2.6} (\cite[Proposition~6.3]{Sm})  A Fano variety (over a field of characteristic zero) with at most rational singularities is of globally F-regular type.
\end{prop}

We will provide in this article some other examples of \emph{globally F-regular type} varieties $Y$, which will be moduli spaces of semi-stable parabolic bundles and moduli spaces of semi-stable parabolic generalized parabolic sheaves. We will construct an open set $X$ of a Fano variety with a morphism $f:X\to Y$ such that $f_*\sO_X=\sO_Y$. It is known by definition and Proposition \ref{prop2.6} that $X$ is of \emph{globally F-regular type}. The following characteristic zero analogy of Lemma \ref{lem2.3} is natural.

\begin{question}\label{question2.7} Let $X\xrightarrow{f} Y$ be a morphism of varieties over a field $K$ of ${\rm char}(K)=0$ such that $\sO_Y\to f_*\sO_X$ is split and $X$ is of
\emph{globally F-regular type}. Is $Y$ a variety of \emph{globally F-regular type} ?
\end{question}

Let $f_*\sO_X\xrightarrow{\beta}\sO_Y$ be a splitting of $\sO_Y\to f_*\sO_X$. Then Question \ref{question2.7} consists: (1)
Can we choose a model $f_A:X_A\to Y_A$ of $f:X\to Y$ such that the $\sO_Y$-homomorphism  $(f_{A_*}\sO_{X_A})\otimes_AK\xrightarrow{\beta}\sO_{Y_A}\otimes_AK$ can be extended to
$f_{A_*}\sO_{X_A}\xrightarrow{\beta_A}\sO_{Y_A}$ ? (2) Is there a dense set of closed point ${\rm Spec}(\overline{k(s)})\to S={\rm Spec}(A)$ such that
$i^*_sf_{A*}\sO_{X_A}=f_{s*}j^*_s\sO_{X_A}$ ?
where $Y_s=Y_A\times_A\overline{k(s)}\xrightarrow{i_s} Y_A$, $X_s=X_A\times_A\overline{k(s)}\xrightarrow{j_s} X_A$ and
$$\xymatrix{
  X_s \ar[d]_{f_s} \ar[r]^{j_s}
                &  X_A\ar[d]^{f_A}  \\
  Y_s \ar[r]^{i_s}
                & Y_A    .}$$

\begin{defn}\label{defn2.8} A morphism $X\xrightarrow{f} Y$ of varieties over a field $K$ of ${\rm char}(K)=0$ is called \emph{$p$-compatible} if there is an integral model
$X_A\xrightarrow{f_A} Y_A$ such that $i^*_sf_{A*}\sO_{X_A}=f_{s*}j^*_s\sO_{X_A}$ for $s\in {\rm Spec}(A)$.
\end{defn}

It is clear that (1) has an affirmative answer when either $f_*\sO_X$ is a coherent $\sO_Y$-module or the splitting map $f_*\sO_X\xrightarrow{\beta}\sO_Y$
is a homomorphism of $\sO_Y$-algebras. (2) is true for flat morphism $f:X\to Y$ with coherent $R\,^if_*\sO_X$ ($i\ge 0$). It is also clear that any affine morphism must be $p$-compatible.
We will give another examples where $X$, $Y$ are open set of GIT quotients and $X\xrightarrow{f}Y$ is induced by a $G$-invariant $p$-compatible morphism $\sR'\xrightarrow{\hat f}\sR$ of parameter spaces. The proof of Proposition \ref{prop2.10} will need the following general observation, which has independent interest.

\begin{lem}\label{lem2.9}
Let $X\to S={\rm Spec}(A)$ be a flat projective morphism, $A$ be an integral $\mathbb{Z}$-algebra of finite type and $G\to S$ be a $S$-flat reductive group scheme
with action on $X$ over $S$. If $L$ is a relative ample line bundle on $X$ linearizing the action of $G$, let
$$X^{ss}(L)\xrightarrow{\pi} X^{ss}(L)//G:=Y$$
be the GIT quotient over $S$. Assume that the geometrically generic fiber of $X^{ss}(L)\to S$ is an irreducible normal variety.
Then there is a dense open set $U\subset S$ such that for any $s\in U$
$$Y\times_S\overline{k(s)}\cong X_s^{ss}(L_s)//G_s$$
where $X_s=X\times_S\overline{k(s)}$ (resp. $G_s=G\times_S\overline{k(s)}$) is the geomerically closed fiber of $X\to S$ (resp. $G\to S$) at ${\rm Spec}(\overline{k(s))}\to S$.
\end{lem}

\begin{proof} Let $X^{ss}(L)\times_S\overline{k(s)}\xrightarrow{\pi_s}Y_s$ be the pullback of $X^{ss}(L)\xrightarrow{\pi}Y$ under the base change
${\rm Spec}(\overline{k(s))}\to S$. By Proposition 7 of \cite{Se},
$$X^{ss}(L)\times_S\overline{k(s)}=X_s^{ss}(L_s).$$
Then there is a unique $\overline{k(s)}$-morphism $X^{ss}_s(L_s)//G_s\xrightarrow{\theta} Y_s$ such that
$$\xymatrix{
   X^{ss}_s(L_s) \ar[dr]^{\pi_s} \ar[r]^{}
                & X_s^{ss}(L_s)//G_s \ar[d]^{\theta}  \\
                & Y_s  }$$
is commutative. Let $\overline{Y_s}:=X^{ss}_s(L_s)//G_s$, $k=\overline{k(s)}$, it is known that $\theta$ induces a bijective map
$\overline{Y_s}(k)\xrightarrow{\theta} Y_s(k)$ on the sets of $k$-points (cf. Proposition 9 (i) of \cite{Se}). By the assumption, geometrically generic fiber of $Y\to S$ is an irreducible normal projective variety. Thus there is a dense open set $U\subset S$ such that any closed point ${\rm Spec}(\overline{k(s))}\to U$ satisfies
(1) $Y_s$ is normal, and (2) the morphism $X_s^{ss}(L_s)\xrightarrow{\pi_s} Y_s$ is generic smooth, where (1) is (iv) of Th{\'e}or{\'e}me (12.2.4) in \cite{Gr} and (2) holds since
$K=Q(A)=k(S)$ is a field of characteristic zero. Then generic smoothness of $\pi_s$ implies the generic smoothness of $\overline{Y_s}\xrightarrow{\theta} Y_s$, which must be
an isomorphism by Zariski main theorem since $Y_s$ is normal.
\end{proof}

Let $(\hat Y,L)$, $(\hat Z,L')$ be polarized projective varieties over an algebraically closed field $K$ of characteristic zero with actions of a reductive group scheme $G$ over $K$, and $\hat Y^{ss}(L)\subset \hat Y$
(resp. $\hat Y^s(L)\subset \hat Y^{ss}(L)$) be the open set of GIT semi-stable (resp. GIT stable) points of $\hat Y$. Then there are projective GIT quotients
\ga{2.1} {\hat Y^{ss}(L)\xrightarrow{\psi} Y:=\hat Y^{ss}(L)//G,\quad \hat Z^{ss}(L')\xrightarrow{\varphi} Z:=\hat Z^{ss}(L')//G.}

\begin{prop}\label{prop2.10} Let $Z$, $Y$ be the GIT quotients in \eqref{2.1}. Assume
\begin{itemize}
\item [(1)] there are $G$-invariant normal open subschemes $\sR\subset \hat Y$, $\sR'\subset \hat Z$ such that $\hat Y^{ss}(L)\subset\sR$, $Z^{ss}(L')\subset \sR'$;
\item [(2)] there is a $G$-invariant $p$-compatible morphism $\sR'\xrightarrow{\hat f}\sR$ such that $\hat f_*\sO_{\sR'}=\sO_{\sR}$;
\item [(3)] there is an $G$-invariant open set $W\subset Z^{ss}(L')$ such that
$${\rm Codim}(\sR'\setminus W)\ge 2, \quad \hat X=\varphi^{-1}\varphi(\hat X)$$ where $\hat X=W\cap \hat f^{-1}(\hat Y^{ss}(L))$.\end{itemize}
If $Z$ is of globally F-regular type. Then so is $Y$.
\end{prop}

\begin{proof} Let $X=\varphi(\hat X)\subset Z$, which is an open set of $Z$ since
$$\varphi(Z^{ss}(L')\setminus\hat X)=Z\setminus X$$
by the condition $\varphi^{-1}(X)=\hat X$ and that $Z^{ss}(L')\setminus\hat X$ is a $G$-invariant closed subset. There is a morphism
$X\xrightarrow{f} Y$ such that
$$\xymatrix{
  \hat X \ar[d]_{\hat f|_{\hat X}} \ar[r]^{\varphi}
                &  X\ar[d]^{f}  \\
  \hat Y^{ss}(L) \ar[r]^{\psi}
                & Y            }$$ is commutative. For any open set $U\subset Y$, since $\hat f_*\sO_{\sR'}=\sO_{\sR}$, we have
$$\aligned\sO_Y(U)&=\sO_{\sR}(\psi^{-1}(U))^{inv.}=\sO_{\sR'}(\hat f^{-1}\psi^{-1}(U))^{inv.}\\
&=\sO_{\sR'}(W\cap \hat f^{-1}\psi^{-1}(U))^{inv.}=\sO_{\hat X}(\hat f|_{\hat X}^{-1}\psi^{-1}(U))^{inv.}\\&=
\sO_{\hat X}(\varphi^{-1}f^{-1}(U))^{inv.}=\sO_X(f^{-1}(U))=f_*\sO_X(U)\endaligned$$
where the third equality holds because $\hat f^{-1}\psi^{-1}(U)\setminus W\cap \hat f^{-1}\psi^{-1}(U)=\hat f^{-1}\psi^{-1}(U)\cap (\sR'\setminus W)$
has codimension at least two. Thus we have
\ga{2.2} {\sO_Y=f_*\sO_X,\,\,\text{where $X$ is of globally F-regular type.}}
To show that $Y$ is of globally F-regular type,  it is enough to show that the morphism $X\xrightarrow{f} Y$ is
$p$-compatible.

Let $(\hat Y_A,\sL)$, $(\hat Z_A,\sL')$ be integral models of $(\hat Y,L)$, $(\hat Z,L')$ with actions of a reductive group scheme $G_A$ over $S={\rm Spec}(A)$, and $\hat Y_A^{ss}(\sL)\subset \hat Y_A$
(resp. $\hat Y_A^s(\sL)\subset \hat Y_A^{ss}(\sL)$) be the open subscheme of GIT semi-stable (resp. GIT stable) points of $\hat Y_A$. Then there are GIT quotients
$$\hat Y_A^{ss}(\sL)\xrightarrow{\psi_A} Y_A:=\hat Y_A^{ss}(\sL)//G_A,\quad \hat Z_A^{ss}(\sL')\xrightarrow{\varphi_A} Z_A:=\hat Z_A^{ss}(\sL')//G_A,$$
which are projective over $S={\rm Spec}(A)$ and $\psi_A$, $\varphi_A$ are surjective $G_A$-invariant affine morphisms (cf. Theorem 4 of \cite{Se}).

We can choose $G_A$-invariant open subschemes
$\sR_A\subset \hat Y_A$, $\sR'_A\subset \hat Z_A$, $W_A\subset Z_A^{ss}(\sL')$, $X_A\subset Z_A$ and a $G_A$-invariant morphism $\sR_A'\xrightarrow{\hat f_A}\sR_A$ such that $\hat Y_A^{ss}(\sL)\subset\sR_A$, $Z_A^{ss}(\sL')\subset \sR_A'$,
$\hat f_{A*}\sO_{\sR_A'}=\sO_{\sR_A}$. Let
$$\hat X_A= \varphi_A^{-1}(X_A), \,\,\sR_s'=\sR'_A\times_A \overline{k(s)},\,\, \sR_s=\sR_A\times_A\overline{k(s)},$$ and $\hat f_s=\hat f_A\otimes \overline{k(s)}$ ($\forall\, s\in S$). Then we have $\hat f_{s*}\sO_{\sR_s'}=\sO_{\sR_s}$,
\ga{2.3} {{\rm Codim}(\sR_s'\setminus W_s)\ge 2,\,\, \hat X_s=W_s\cap \hat f_s^{-1}(\hat Y_A^{ss}(\sL)\times_A\overline{k(s)})} (by shrinking $S$) where $W_s=W_A\times_A\overline{k(s)}$, $\hat X_s=\hat X_A\times_A\overline{k(s)}$ and $$\hat Y_A^{ss}(\sL)\times_A\overline{k(s)}=\hat Y_s^{ss}(\sL_s),\quad \hat Z_A^{ss}(\sL')\times_A\overline{k(s)}=\hat Z_s^{ss}(\sL'_s)$$ (cf. Proposition 7 of \cite{Se}). Then, by Lemma \ref{lem2.9}, we have
$$Z_s=Z^{ss}_s(\sL'_s)//G_s, \quad Y_s=Y_s^{ss}(\sL_s)//G_s.$$ Thus, for any open sets $U\subset Z_s$, $V\subset Y_s$,  one has
$$\sO_{Z_s}(U)=\sO_{\sR'_s}(\varphi_s^{-1}(U))^{inv.},\quad \sO_{Y_s}(V)=\sO_{\sR_s}(\psi_s^{-1}(V))^{inv.}.$$
Recall $X_s\subset Z_s$, $\varphi_s^{-1}(X_s)=\hat X_s=W_s\cap \hat f_s^{-1}(\hat Y_s^{ss}(\sL_s))$ and consider
$$\xymatrix{
  \hat X_s \ar[d]_{\hat f_s} \quad \ar[r]^{\varphi_s}
                &  X_s\ar[d]^{f_s}  \\
  \hat Y_s^{ss}(\sL_s) \ar[r]^{\psi_s}
                & Y_s            }$$
we have $\sO_{Y_s}(V)=\sO_{\sR_s}(\psi_s^{-1}(V))^{inv.}=\sO_{\sR'_s}(\hat f_s^{-1}(\psi_s^{-1}(V)))^{inv.}$ since $\hat f_{s*}\sO_{\sR'_s}=\sO_{\sR_s}$.
Because the codimension of $$\hat f_s^{-1}(\psi_s^{-1}(V))\setminus W_s\cap \hat f_s^{-1}(\psi_s^{-1}(V))=\hat f_s^{-1}(\psi_s^{-1}(V))\cap (\sR'_s\setminus W_s)$$
is at least two, we have
$$\aligned \sO_{Y_s}(V)&=\sO_{\hat X_s}(\hat f_s^{-1}(\psi_s^{-1}(V)))^{inv.}=\sO_{\hat X_s}(\varphi_s^{-1}f_s^{-1}(V))^{inv.}\\&=
\sO_{X_s}(f^{-1}_s(V))=(f_s)_*\sO_{X_s}(V).\endaligned$$
Thus $\sO_{Y_s}=(f_s)_*\sO_{X_s}$, which implies that $f:X\to Y$ is $p$-compatible and $Y$ is of globally F-regular type since $X$ is so.

\end{proof}

\section{Moduli spaces of parabolic bundles and generalized parabolic sheaves}

In this section, we prove that moduli spaces of parabolic bundles and generalized parabolic sheaves with a fixed
determinant on a smooth curve are of globally F-regular type.

Let $C$ be an irreducible projective curve of genus $g\ge 0$ over an
algebraically closed field $K$ of characteristic zero, which has at most
one node $x_0\in C$. Let $I$ be a finite set of smooth points of $C$, and
$E$ be a coherent sheaf of rank $r$ and degree $d$ on $C$ (the rank
$r(E)$ is defined to be dimension of $E_{\xi}$ at generic point
$\xi\in C$, and $d=\chi(E)-r(1-g)$).

\begin{defn}\label{defn3.1} By a quasi-parabolic structure of $E$ at a
smooth point $x\in C$, we mean a choice of flag of quotients
$$E_x=Q_{l_x+1}(E)_x\twoheadrightarrow
Q_{l_x}(E)_x\twoheadrightarrow\cdots\cdots\twoheadrightarrow
Q_1(E)_x\twoheadrightarrow Q_0(E)_x=0$$ of the fibre $E_x$, $n_i(x)={\rm
dim}(ker\{Q_i(E)_x\twoheadrightarrow Q_{i-1}(E)_x\})$ ($1\le i\le l_x+1$)
are called type of the flags. If, in addition, a sequence of integers
$$0\leq a_1(x)<a_2(x)<\cdots
<a_{l_x+1}(x)< k$$ are given, we call that $E$ has a parabolic
structure of type $$\vec n(x)=(n_1(x),n_2(x),\cdots,n_{l_x+1}(x))$$ and
weight $\vec a(x)=(a_1(x),a_2(x),\cdots,a_{l_x+1}(x))$ at $x\in C$.
\end{defn}

\begin{defn}\label{defn3.2} For any subsheaf $F\subset E$, let $Q_i(E)_x^F\subset
Q_i(E)_x$ be the image of $F$ and $n_i^F={\rm
dim}(ker\{Q_i(E)_x^F\twoheadrightarrow Q_{i-1}(E)_x^F\})$. Let
$${\rm par}\chi(E):=\chi(E)+\frac{1}{k}\sum_{x\in
I}\sum^{l_x+1}_{i=1}a_i(x)n_i(x),$$
$${\rm par}\chi(F):=\chi(F)+\frac{1}{k}\sum_{x\in
I}\sum^{l_x+1}_{i=1}a_i(x)n^F_i(x).$$
Then $E$ is called semistable (resp., stable) for $\omega=(k, \{\vec n(x),\,\,\vec a(x)\}_{x\in I})$ if for any
nontrivial $E'\subset E$ such that $E/E'$ is torsion free,
one has
$${\rm par}\chi(E')\leq
\frac{{\rm par}\chi(E)}{r}\cdot r(E')\,\,(\text{resp., }<).$$
\end{defn}

\begin{thm}[Theorem X1 of \cite{NR} or Theorem 2.13 of \cite{Su3} for arbitrary rank]\label{thm3.3} There
exists a seminormal projective variety
$$\sU_{C,\,\omega}:=\sU_C(r,d,
\{k,\vec n(x),\vec a(x)\}_{x\in I}),$$ which is the coarse moduli
space of $s$-equivalence classes of semistable parabolic sheaves $E$
of rank $r$ and $\chi(E)=\chi=d+r(1-g)$ with parabolic structures of type
$\{\vec n(x)\}_{x\in I}$ and weights $\{\vec a(x)\}_{x\in I}$ at
points $\{x\}_{x\in I}$. If $C$ is smooth, then it is normal, with
only rational singularities.
\end{thm}

Recall the construction of $\sU_{C,\,\omega}=\sU_C(r,d,\omega)$. Fix a line bundle $\sO(1)=\sO_C(c\cdot y)$ on $C$ of ${\rm deg}(\sO(1))=c$, let
$\chi=d+r(1-g)$, $P$ denote the polynomial $P(m)=crm+\chi$,
$\sO_C(-N)=\sO(1)^{-N}$ and $V=\Bbb C^{P(N)}$. Let $\bold Q$ be the Quot scheme of quotients $V\otimes\sO_{C}(-N)\to F\to 0$ (of rank
$r$ and degree $d$) on $C$. Thus there is on $C\times\bold Q$ a universal quotient
$$V\otimes\sO_{C\times\bold Q}(-N)\to \sF\to 0.$$
Let $\sF_x=\sF|_{\{x\}\times\bold Q}$ and $Flag_{\vec n(x)}(\sF_x)\to\bold Q$ be the relative flag scheme of type $\vec n(x)$. Let
$$\sR=\underset{x\in I}{\times_{\bold Q}}Flag_{\vec n(x)}(\sF_x)\to \bold Q,$$
on which reductive group ${\rm SL}(V)$ acts. The data $\omega=(k, \{\vec n(x),\,\,\vec a(x)\}_{x\in I})$, more precisely, the
weight $(k,\{\vec a(x)\}_{x\in I})$ determines a polarisation
$$\Theta_{\sR,\omega}=({\rm
det}R\pi_{\sR}\sE)^{-k}\otimes\bigotimes_{x\in I}
\lbrace\bigotimes^{l_x}_{i=1} {\rm det}(\sQ_{\{x\}\times
\sR,i})^{d_i(x)}\rbrace\otimes\bigotimes_q{\rm
det}(\sE_y)^{\ell}$$
on $\sR$ such that the open set $\sR^{ss}_{\omega}$ (resp. $\sR^s_{\omega}$) of
GIT semistable (resp. GIT stable) points are precisely the set of semistable (resp. stable) parabolic sheaves on $C$ (see \cite{Su3}), where $\sE$ is the pullback of $\sF$ (under
 $C\times\sR\to C\times \bold Q$), ${\rm
det}R\pi_{\sR}\sE$ is determinant line bundle of cohomology,
$$\sE_x=\sQ_{\{x\}\times \sR,l_x+1}\twoheadrightarrow\sQ_{\{x\}\times \sR,l_x}\twoheadrightarrow \sQ_{\{x\}\times \sR,l_x-1}
\twoheadrightarrow\cdots\twoheadrightarrow \sQ_{\{x\}\times
\sR,1}\twoheadrightarrow0$$ are universal quotients on $\sR$ of type $\vec n(x)$, $d_i(x)=a_{i+1}(x)-a_i(x)$ and
$$\ell:=\frac{k\chi-\sum_{x\in I}\sum^{l_x}_{i=1}d_i(x)r_i(x)}{r}.$$
Then $\sU_{C,\,\omega}$ is the GIT quotient $\sR^{ss}_{\omega}\xrightarrow{\psi} \sU_{C,\,\omega}:=\sU_C(r,d, \omega)$ and $\Theta_{\sR^{ss},\omega}$
descends to an ample line bundle $\Theta_{\sU_{C,\,\omega}}$ on $\sU_{C,\,\omega}$ when $\ell$ is an integer.

\begin{defn}\label{defn3.4} When $C$ is a smooth projective curve, let
$${\rm Det}: \sU_{C,\,\omega}\to J^d_C,\quad E\mapsto {\rm det}(E):=\bigwedge^rE$$
be the determinant map. Then, for any $L\in J^d_C$, the fiber
$${\rm Det}^{-1}(L):=\sU_{C,\,\omega}^L$$ is called moduli space of semistable parabolic bundles with a fixed determinant.
\end{defn}

Let $\sR^L_F\subset \sR$ be the sub-scheme of locally free sheaves with a fixed determinant $L$, and $(\sR^{ss}_{\omega})^L\subset\sR^{ss}_{\omega}$, $\,(\sR^{s}_{\omega})^L\subset\sR^s_{\omega}$
be the closed subsets of locally free sheaves with the fixed determinant $L$. Then $\sU_{C,\,\omega}^L$ is the GIT quotient
$(\sR^{ss}_{\omega})^L\xrightarrow{\psi}(\sR^{ss}_{\omega})^L//{\rm SL}(V):=\sU_{C,\,\omega}^L$. The proof of globally F-regular type of $\sU_{C,\,\omega}^L$ needs essentially the following two results.

\begin{prop}\label{prop3.5} Let $|{\rm I}|$ be the number
of parabolic points. Then, for any data $\omega=(k, \{\vec n(x),\,\,\vec a(x)\}_{x\in I})$, we have
\begin{itemize}
 \item[(1)]  $\,\,{\rm Codim}((\sR_{\omega}^{ss})^L\setminus (\sR_{\omega}^s)^L)\ge (r-1)(g-1)+\frac{1}{k}|{\rm
 I}|$,
\item[(2)]  $\,\,{\rm Codim} (\sR^L_F\setminus(\sR_{\omega}^{ss})^L)>(r-1)(g-1)+\frac{1}{k}|{\rm
I}|$.
\end{itemize}
\end{prop}

\begin{proof} This is in fact Proposition 5.1 of \cite{Su1} where we did not fix determinant and the term $\frac{1}{k}|{\rm
 I}|$ was omitted. However, the proof also works for the case of fixed determinant.
\end{proof}

\begin{prop}\label{prop3.6} Let $\omega_c=(2r, \{\vec n(x),\,\,\vec a_c(x)\}_{x\in I})$, where
$$\vec a_c(x)=(\bar a_1(x),\bar a_2(x),\cdots,\bar a_{l_x+1}(x))$$
satisfy $\bar a_{i+1}(x)-\bar a_i(x)=n_i(x)+n_{i+1}(x)$ ($1\le i\le l_x$). Then, when
\ga{3.1} {(r-1)(g-1)+\frac{|I|}{2r}\ge 2,}
the moduli space $\sU^L_{C,\,\omega_c}=(\sR^{ss}_{\omega_c})^L//{\rm SL}(V)$ is a normal Fano variety with only rational singularities.
\end{prop}

\begin{proof} It is in fact a reformulation of Proposition 2.2 of \cite{Su1} where a formula of anti-canonical bundle $\omega^{-1}_{\sR_F}$ (thus
a formula of $\omega^{-1}_{\sR^L_F}$) was given (see also Proposition 4.2 of \cite{Su3} for a tidier formula). The line bundle $\omega^{-1}_{\sR^L_F}$
is precisely determined by the data $\omega_c$ and descends to an ample line bundle $\Theta_{\sU^L_{C,\,\omega_c}}$, which is precisely $\omega^{-1}_{\sU^L_{C,\,\omega_c}}$ when $${\rm Codim}((\sR_{\omega_c}^{ss})^L\setminus (\sR_{\omega_c}^s)^L)\ge 2$$ by a result of F. Knop (see \cite{Kn}).
Thus we are done by the condition \eqref{3.1} and (1) of Proposition \ref{prop3.5}.
\end{proof}

\begin{thm}\label{thm3.7} The moduli spaces $\sU_{C,\,\omega}^L$ are of globally F-regular type. If Jacobian $J^0_C$ of $C$ is of F-split type,
so is $\sU_{C,\,\omega}$.
\end{thm}

\begin{proof} Choose a subset $I'\subset C$  such that $I'\cap I=\emptyset$ and
\ga{3.2} {(r-1)(g-1)+\frac{|I|+|I'|}{2r}\ge 2.}
Let $$\sR'=\underset{x\in I\cup I'}{\times_{\mathbf{Q}}}
Flag_{\vec n(x)}(\sF_x)=\sR\times_{\mathbf{Q}}\left(\underset{x\in I'}{\times_{\mathbf{Q}}}
Flag_{\vec n(x)}(\sF_x)\right)\xrightarrow{\hat f} \sR$$ be the projection, $\hat Y\subset\sR$ be the Zariski closure of $\sR^L_F$ and
$$\hat Z={\hat f}^{-1}(\hat Y)\subset \sR',\quad {\sR'}^L_F={\hat f}^{-1}(\sR^L_F)\subset\hat Z.$$
Then ${\sR'}^L_F\subset \hat Z$, $\sR^L_F\subset \hat Y$ are normal (in fact, smooth) ${\rm SL}(V)$-invariant open sub-schemes such that
$\hat Y^{ss}_{\omega}=(\sR^{ss}_{\omega})^L\subset \sR^L_F$,  $\hat Z^{ss}_{\omega'}=(\sR'^{ss}_{\omega'})^L\subset \sR'^L_F$
holds for any polarizations determined by data $\omega$, $\omega'$. It is clear that
$ \sR'^L_F\xrightarrow{\hat f}\sR^L_F$ is a flag bundle and $p$-compatible with $\hat f_*\sO_{\sR'^L_F}=\sO_{\sR^L_F}$. Thus
$\sR'^L_F\subset \hat Z$, $\sR^L_F\subset \hat Y$, $\sR'^L_F\xrightarrow{\hat f}\sR^L_F$
satisfy the conditions (1) and (2) of Proposition \ref{prop2.10}. To verify condition (3) in Proposition \ref{prop2.10}, let
$$W:=\hat Z^{s}_{\omega'}=(\sR'^{s}_{\omega'})^L\subset \sR'^L_F,\quad \hat X=\hat f^{-1}((\sR^{ss}_{\omega})^L)\cap W,$$
$\hat Z^{ss}_{\omega'}\xrightarrow{\varphi} Z:=\hat Z^{ss}_{\omega'}//{\rm SL}(V)$ and $X=\varphi(\hat X)\subset Z$. It is clear that
$$\hat X=\varphi^{-1}(X).$$
If we choose $\omega'=(2r, \{\vec n(x),\,\,\vec a_c(x)\}_{x\in I\cup I'})$ in Proposition \ref{prop3.6}, then
$${\rm Codim}({\sR'}^L_F\setminus W)\ge (r-1)(g-1)+\frac{|I|+|I'|}{2r}\ge 2$$
by Proposition \ref{prop3.5}, and $Z$ is a normal Fano variety with only rational singularities.
Thus $Z$ is of globally F-regular type by Proposition \ref{prop2.6}, so is $\sU_{C,\,\omega}^L=(\sR^{ss}_{\omega})^L//{\rm SL}(V)=\hat Y^{ss}_{\omega}//{\rm SL}(V)$
by Proposition \ref{prop2.10}.

If $J^0_C$ is of F-split type, so is $J^0_C\times \sU^L_{C,\,\omega}$. We have a $r^{2g}$-fold covering
$$J^0_C\times \sU^L_{C,\,\omega}\xrightarrow{f} \sU_{C,\,\omega}, \quad f(\sL_0, E)=\sL_0\otimes E$$
which implies that $\sU_{C,\,\omega}$ is of F-split type.
\end{proof}

Let $\{x_1,\,x_2\}\subset C\setminus I$ be two different points, a \emph{generalized parabolic sheaf} (GPS) $(E,Q)$ of rank $r$ and degree $d$ on $C$ consists of
a sheaf $E$ of degree $d$ on $C$, torsion free of rank $r$ outside
$\{x_1,x_2\}$ with parabolic structures at the points of $I$ and an $r$-dimensional quotient
$$E_{x_1}\oplus E_{x_2}\xrightarrow{q} Q\to 0.$$

\begin{defn}\label{defn3.8} A GPS $(E,Q)$ on an irreducible smooth curve $C$ is called \emph{semistable} (resp.,
\emph{stable}), if for every nontrivial subsheaf $E'\subset E$ such that
$E/E'$ is torsion free outside $\{x_1,x_2\},$ we have
$$par\chi(E')-dim(Q^{E'})\leq
r(E')\cdot\frac{par\chi(E)-dim(Q)}{r(E)} \,\quad (\text{resp.,
$<$}),$$ where $Q^{E'}=q(E'_{x_1}\oplus E'_{x_2})\subset Q.$
\end{defn}

\begin{thm}[Theorem X2 of \cite{NR} or Theorem 2.24 of \cite{Su3} for arbitrary rank]\label{thm3.9} For any data $\omega=(k, \{\vec n(x),\,\,\vec a(x)\}_{x\in I})$, there exists
a normal projective variety $\sP_{\omega}$ with at most rational singularities, which is the
coarse  moduli space of $s$-equivalence classes of semi-stable GPS on $C$ with parabolic structures
at the points of $I$ given by the data $\omega$.
\end{thm}

Recall the construction of $\sP_{\omega}$. Let $Grass_r(\sF_{x_1}\oplus\sF_{x_2})\to \bold Q$ and
$$\widetilde{\sR}=Grass_r(\sF_{x_1}\oplus\sF_{x_2})\times_{\bold
Q}\sR\xrightarrow{\rho} \sR.$$  $\omega=(k, \{\vec n(x),\,\,\vec a(x)\}_{x\in I})$ determines
a polarization, which linearizes the ${\rm SL}(V)$-action on $\wt\sR$, such that the open set $\wt{\sR}^{ss}_{\omega}$ (resp. $\wt{\sR}^s_{\omega}$) of
GIT semistable (resp. GIT stable) points are precisely the set of semistable (resp. stable) GPS on $C$ (see \cite{Su3}). Then $\sP_{\omega}$ is the GIT quotient
\ga{3.3} {\wt{\sR}^{ss}_{\omega}\xrightarrow{\psi}\wt{\sR}^{ss}_{\omega}//{\rm SL}(V):=\sP_{\omega}.}

\begin{nota}\label{nota3.10} Let $\sH\subset\wt{\sR}$ be the open subscheme
parametrising the generalised parabolic sheaves $E=(E,E_{x_1} \oplus
E_{x_2}\xrightarrow{q}Q)$ satisfying \begin{itemize}
\item [(1)] the torsion ${\rm Tor}\,E$ of $E$ is
supported on $\{x_1,x_2\}$ and $$q:({\rm Tor}\,E)_{x_1}\oplus ({\rm
Tor}\,E)_{x_2}\hookrightarrow Q$$
\item [(2)] if $N$ is large enough, then
$H^1(E(N)(-x-x_1-x_2))=0$ for all $E$ and $x\in C$.
\end{itemize}
\end{nota}
Then $\sH$ is reduced, normal, Gorenstein with at most rational singularities (see Proposition 3.2 and Remark 3.1 of \cite{Su1}). Moreover,
for any data $\omega$, we have $\wt{\sR}^{ss}_{\omega}\subset \sH$ and, by Lemma 5.7 of \cite{Su1}, there is a morphism ${\rm Det}_{\sH}:\sH\to J^d_C$
which extends determinant morphism on open set $\wt\sR_F\subset\sH$ of locally free sheaves, and induces a flat morphism
\ga{3.4} {{\rm Det}: \sP_{\omega}\to J^d_C.}

\begin{nota}\label{nota3.11} For $L\in J^d_C$, let $\sH^L={\rm Det}^{-1}(L)\subset\sH$,
$$\wt\sR_F^L={\rm Det}^{-1}(L)\subset\wt\sR_F, \quad (\wt\sR_{\omega}^{ss})^L={\rm Det}^{-1}(L)\subset \wt\sR_{\omega}^{ss}.$$
Then $\sP^L_{\omega}={\rm Det}^{-1}(L)\subset \sP_{\omega}$ is the GIT quotient
$$(\wt\sR_{\omega}^{ss})^L\xrightarrow{\psi}\sP^L_{\omega}=(\wt\sR_{\omega}^{ss})^L//{\rm SL}(V).$$
\end{nota}

\begin{prop}[Proposition 5.2 of \cite{Su1}]\label{prop3.12}
Let $\sD_1^f=\hat\sD_1\cup\hat\sD_1^t$
and $\sD_2^f=\hat\sD_2\cup\hat\sD_2^t$, where $\hat\sD_i\subset \wt{\sR}$ is the Zariski closure of
$\hat\sD_{F,\,i}\subset\wt{\sR}_F$ consisting of $(E,Q)\in\wt{\sR}_F$ that $E_{x_i}\to Q$ is not an isomorphism, and
${\hat\sD}_1^t\subset\wt{\sR}$ (rep. ${\hat\sD}_2^t\subset\wt{\sR}$) consists of $(E,Q)\in\wt{\sR}$ such that $E$ is not locally
free at $x_2$ (resp. at $x_1$). Then \begin{itemize}
\item [(1)] ${\rm Codim}(\sH^L\setminus(\wt\sR_{\omega}^{ss})^L)>(r-1)g+\frac{|I|}{k};$
\item [(2)] the complement in $(\wt\sR_{\omega}^{ss})^L\setminus\{\sD_1^f\cup\sD^f_2\}$ of the set
$\wt\sR_{\omega}^{s}$ of stable points has codimension $\ge
(r-1)g+\frac{|I|}{k}$.
\item [(3)] ${\rm Codim}((\wt\sR^{ss}_{\omega})^L\setminus W_{\omega})\ge (r-1)g+\frac{|I|}{k}$, $W_{\omega}\subset (\wt\sR^{ss}_{\omega})^L$ defined by
$$W_{\omega}:=\left\{ (E, Q)\in (\wt\sR_{\omega}^{ss})^L\bigg|
\begin{aligned}
& \forall \,\,E'\subset E \ \ \text{with}\ \ 0<r(E')< r, \ \text{we have} \\& \frac{par\chi(E')-dim(Q^{E'})}{r(E')}<\frac{par\chi(E)-dim(Q)}{r(E)}
\end{aligned}\right\}.$$
\end{itemize}
\end{prop}

\begin{proof} The statements (1) and (2) are contained in Proposition 5.2 of \cite{Su1} (where the term $\frac{|I|}{k}$ was omitted). The proof of Proposition 5.2 (2) in \cite{Su1} implies statement (3) here.

\end{proof}

\begin{prop}\label{prop3.13} Let $\omega_c=(2r, \{\vec n(x),\,\,\vec a_c(x)\}_{x\in I})$ be the data in Proposition \ref{prop3.6}
and $\Theta_{J^d_C}$ be the theta line bundle on $J^d_C$. Assume
\ga{3.5} {(r-1)(g-1)+\frac{|I|}{2r}\ge 2.}
Then there is an ample line bundle $\Theta_{\sP_{\omega_c}}$ on $\sP_{\omega_c}$ such that
$$\omega^{-1}_{\sP_{\omega_c}}=\Theta_{\sP_{\omega_c}}\otimes {\rm Det}^*(\Theta_{J^d_C}^{-1}).$$
In particular, for any $L\in J^d_C$, $\sP^L_{\omega_c}$ is a normal Fano variety with only rational singularities.
\end{prop}

\begin{proof} Let
$V\otimes\sO_{C\times\sH}(-N)\to \sE\to 0$, $\,\,\sE_{x_1}\oplus\sE_{x_2}\to \sQ\to 0$ and
$$\{\,\,\sE_{\{x\}\times \sH}=\sQ_{\{x\}\times \sH,\,l_x+1}\twoheadrightarrow\sQ_{\{x\}\times \sH,\,l_x}
\twoheadrightarrow \cdots\twoheadrightarrow \sQ_{\{x\}\times\,\sH,1}
\twoheadrightarrow0\,\,\}_{x\in I}$$ be the universal quotients and universal flags. Let $\omega_{C}=\sO(\sum_qq)$ and
$$\Theta_{J^d_C}=(detR\pi_{J^d_{C}}\sL)^{-2}\otimes\sL_{x_1}^r\otimes\sL_{x_2}^r\otimes\sL_y^{2\chi-2r}\otimes\bigotimes_q\sL_q^{r-1}$$
where $\sL$ is the universal line bundle on $C\times
J^d_{C}$. Then we have
$$\aligned&\omega^{-1}_{\sH}=(det\,R\pi_{\sH}\sE)^{-2r}\otimes\\&
\bigotimes_{x\in I}\left\{(det\,\sE_x)^{n_{l_x+1}-r}
\otimes\bigotimes^{l_x}_{i=1}(det\,\sQ_{x,i})
^{n_i(x)+n_{i+1}(x)}\right\}\otimes(det\,\sQ)^{2r}\\&
\otimes(det\,\sE_y)^{2\chi-2r}\otimes{\rm Det}_{\sH}^*(\Theta_{J^d_{C}}^{-1}):=\hat\Theta_{\omega_c}\otimes {\rm Det}_{\sH}^*(\Theta_{J^d_{C}}^{-1})
\endaligned$$ by Proposition 3.4 of \cite{Su1}, and $\hat\Theta_{\omega_c}$ descends to an ample
line bundle $\Theta_{\sP_{\omega_c}}$ on $\sP_{\omega_c}$ (see Lemma 2.3 of \cite{Su1}). Thus
$$(\psi_*\omega_{\wt{\sR}^{ss}_{\omega_c}}^{-1})^{inv.}=\Theta_{\sP_{\omega_c}}\otimes{\rm Det}^*(\Theta_{J^d_C}^{-1}).$$
When condition \eqref{3.5} holds, the lower bounds in Proposition \ref{prop3.5} and Proposition \ref{prop3.12} are at least two. Thus
Lemma 5.6 of \cite{Su1} is applicable (where assumption $g\ge 2$ in Lemma 5.6 of \cite{Su1} is replaced by condition \eqref{3.5})
and $(\psi_*\omega_{\wt{\sR}^{ss}_{\omega_c}}^{-1})^{inv.}=\omega^{-1}_{\sP_{\omega_c}}$.
\end{proof}

\begin{thm}\label{thm3.14} For any data $\omega=(k, \{\vec n(x),\,\,\vec a(x)\}_{x\in I})$, the moduli space
$\sP^L_{\omega}$ is of globally F-regular type.
\end{thm}

\begin{proof} Choose a finite subset $I'\subset C\setminus I$ satisfying \eqref{3.5}. Recall that
$$\sR=\underset{x\in I}{\times_{\bold Q}}Flag_{\vec n(x)}(\sF_x),\quad \sR'=\underset{x\in I\cup I'}{\times_{\mathbf{Q}}}
Flag_{\vec n(x)}(\sF_x)\xrightarrow{\hat f} \sR$$ be the projection and
$\widetilde{\sR}=Grass_r(\sF_{x_1}\oplus\sF_{x_2})\times_{\bold Q}\sR\xrightarrow{\rho} \sR$. Let
\ga{3.6} {\wt\sR':=Grass_r(\sF_{x_1}\oplus\sF_{x_2})\times_{\bold Q}\sR'\xrightarrow{\hat f} \wt\sR.}
Then, on $\sH^L\subset \wt\sR$,  it is clear that $(\sH')^L:={\hat f}^{-1}(\sH^L)\xrightarrow{\hat f} \sH^L$ is a ${\rm SL}(V)$-invariant and
$p$-compatible morphism such that $\hat f_*\sO_{(\sH')^L}=\sO_{\sH^L}$.

For any data $\omega=(k, \{\vec n(x),\,\,\vec a(x)\}_{x\in I})$, $\omega_c=(2r, \{\vec n(x),\,\,\vec a_c(x)\}_{x\in I\cup I'})$, we have
$(\wt\sR_{\omega}^{ss})^L\subset\sH^L$, $\,(\wt\sR_{\omega_c}^{\prime\,ss})^L\subset (\sH')^L$. Recall
$$(\wt\sR_{\omega}^{ss})^L\xrightarrow{\psi}\sP^L_{\omega}:=Y,\quad (\wt\sR_{\omega_c}^{\prime\,ss})^L\xrightarrow{\varphi}\sP^L_{\omega_c}:=Z.$$
To apply Proposition \ref{prop2.10}, let $W=W_{\omega_c}\subset (\wt\sR_{\omega_c}^{\prime\,ss})^L$ and
$$\hat X=W\cap\hat f^{-1}((\wt\sR_{\omega}^{ss})^L).$$
By Proposition \ref{prop3.12},  ${\rm Codim}((\sH')^L\setminus W)\ge (r-1)g+\frac{|I|+|I'|}{2r}\ge 2$. Thus it is enough to
check the condition that $\hat X=\varphi^{-1}\varphi(\hat X)$. This is equivalent (see Remark 1.2 of \cite{Su1}) to prove that
\ga{3.7}{\forall\,\,(E, Q)\in (\wt\sR_{\omega_c}^{\prime\,ss})^L,\quad (E,Q)\in \hat X\,\,\,\Leftrightarrow\,\,\, gr(E,Q)\in \hat X.}
In fact, for any $(E, Q)\in (\wt\sR_{\omega_c}^{\prime\,ss})^L$, it is clear that we have
$$(E, Q)\in W \,\,\,\Leftrightarrow\,\,\,gr(E,Q)=(\wt E, \wt Q)\oplus (\,_{x_1}\tau_1\oplus\,_{x_2}\tau_2, \tau_1\oplus\tau_2)\in W$$
where $(\wt E,\wt Q)$ is a stable GPS (see Definition 1.5 of \cite{Su1}). Thus either
$$0\to (\,_{x_1}\tau_1\oplus\,_{x_2}\tau_2, \tau_1\oplus\tau_2)\to (E,Q)\to (\wt E,\wt Q)\to 0$$
or $0\to (E',Q')\to (E, Q)\to (\,_{x_i}\mathbb{C},\mathbb{C})\to 0$. Then $(E,Q)$ is semi-stable (respect to $\omega$) if and only if
$(\wt E,\wt Q)$ is semi-stable (respect to $\omega$). Thus \eqref{3.7} is proved and we are done.
\end{proof}

When $C=C_1\cup C_2$ is reducible with two smooth irreducible
components $C_1$ and $C_2$ of genus $g_1$ and $g_2$ meeting at only
one point $x_0$ (which is the only node of $C$), we fix an ample
line bundle $\sO(1)$ of degree $c$ on $C$ such that
$deg(\sO(1)|_{C_i})=c_i>0$ ($i=1,2$). For any coherent sheaf $E$,
$P(E,n):=\chi(E(n))$ denotes its Hilbert polynomial, which has degree
$1$. We define the rank of $E$ to be
$$r(E):=\frac{1}{deg(\sO(1))}\cdot \lim \limits_{n\to\infty}\frac{P(E,n)}
{n}.$$ Let $r_i$ denote the rank of the restriction of $E$ to $C_i$
($i=1,2$), then
$$P(E,n)=(c_1r_1+c_2r_2)n+\chi(E),\quad r(E)=
\frac{c_1}{c_1+c_2}r_1+\frac{c_2}{c_1+c_2}r_2.$$ We say that $E$ is
of rank $r$ on $X$ if $r_1=r_2=r$, otherwise it will be said of rank
$(r_1,r_2)$.

Fix a finite set $I=I_1\cup I_2$ of smooth points on $C$, where
$I_i=\{x\in I\,|\,x\in C_i\}$ ($i=1,2$), and parabolic data $\omega=\{k,\vec
n(x),\vec a(x)\}_{x\in I}$ with
$$\ell:=\frac{k\chi-\sum_{x\in I}\sum^{l_x}_{i=1}d_i(x)r_i(x)}{r}$$
(recall $d_i(x)=a_{i+1}(x)-a_i(x)$, $r_i(x)=n_1(x)+\cdots+n_i(x)$). Let
\ga{3.8}{n^{\omega}_j=\frac{1}{k}\left(r\frac{c_j}{c_1+c_2}\ell+\sum_{x\in
I_j}\sum^{l_x}_{i=1}d_i(x)r_i(x)\right)\,\,\,(j=1,\,\,2).}

\begin{defn}\label{defn3.15} For any coherent sheaf $F$ of rank $(r_1,r_2)$, let
$$m(F):= \frac{r(F)-r_1}{k}\sum_{x\in I_1}a_{l_x+1}(x)+
\frac{r(F)-r_2}{k}\sum_{x\in I_2}a_{l_x+1}(x),$$ the modified
parabolic Euler characteristic and slop of $F$ are
$${\rm par}\chi_m(F):={\rm par}\chi(F)+m(F),\quad {\rm par}\mu_m(F):=\frac{{\rm par}\chi_m(F)}{r(F)}.$$
A parabolic sheaf $E$ is called semistable (resp. stable) if, for
any subsheaf $F\subset E$ such $E/F$ is torsion free, one has, with
the induced parabolic structure,
$${\rm par}\chi_m(F)\le \frac{{\rm par}\chi_m(E)}{r(E)}r(F)\quad (resp.<).$$
\end{defn}

\begin{thm}[Theorem 1.1 of \cite{Su2} or Theorem 2.14 of \cite{Su3}]\label{thm3.16} There
exists a reduced, seminormal projective scheme
$$\sU_C:=\sU_C(r,d,\sO(1),
\{k,\vec n(x),\vec a(x)\}_{x\in I_1\cup I_2})$$ which is the coarse
moduli space of $s$-equivalence classes of semistable parabolic
sheaves $E$ of rank $r$ and $\chi(E)=\chi=d+r(1-g)$ with parabolic structures
of type $\{\vec n(x)\}_{x\in I}$ and weights $\{\vec a(x)\}_{x\in
I}$ at points $\{x\}_{x\in I}$. The moduli space $\sU_C$ has at most
$r+1$ irreducible components.
\end{thm}

The normalization of $\sU_C$ is a moduli space of semistable GPS on $\wt C=C_1\bigsqcup C_2$ with parabolic structures at points $x\in I$. Recall

\begin{defn}\label{defn3.17} A GPS $(E,E_{x_1}\oplus E_{x_2}\xrightarrow{q}Q)$ is called semistable (resp.,
stable), if for every nontrivial subsheaf $E'\subset E$ such that
$E/E'$ is torsion free outside $\{x_1,x_2\},$ we have, with the
induced parabolic structures at points $\{x\}_{x\in I}$,
$$par\chi_m(E')-dim(Q^{E'})\leq
r(E')\cdot\frac{par\chi_m(E)-dim(Q)}{r(E)} \,\quad (\text{resp.,
$<$}),$$ where $Q^{E'}=q(E'_{x_1}\oplus E'_{x_2})\subset Q.$
\end{defn}

\begin{thm}[Theorem 2.1 of \cite{Su2} or Theorem 2.26 of \cite{Su3}]\label{thm3.18} For any data $\omega=(\{k,\vec n(x),\,\,\vec a(x)\}_{x\in I_1\cup I_2},\sO(1))$,
the coarse  moduli space $\sP_{\omega}$ of $s$-equivalence classes of semi-stable GPS on $\wt C$ with parabolic structures
at the points of $I$ given by the data $\omega$ is a disjoint union of at most $r+1$ irreducible, normal projective varieties $\sP_{\chi_1,\chi_2}$
( $\chi_1+\chi_2=\chi+r$, $n_j^{\omega}\le\chi_j\le n_j^{\omega}+r$) with at most rational singularities.
\end{thm}

For fixed $\chi_1$, $\chi_2$ satisfying $\chi_1+\chi_2=\chi+r$ and $n_j^{\omega}\le\chi_j\le n_j^{\omega}+r$ ($j=1,\,2$), recall the construction
of $\sP_{\omega}=\sP_{\chi_1,\chi_2}$. Let
$$P_i(m)=c_irm+\chi_i,\quad \sW_i=\sO_{C_i}(-N),\quad V_i=\Bbb
C^{P_i(N)}$$ where
$\sO_{C_i}(1)=\sO (1)|_{C_i}$ has degree $c_i$. Consider the Quot schemes $\textbf{Q}_i=Quot(V_i\otimes\sW_i,
P_i)$, the universal quotient
$V_i\otimes\sW_i\to \sF^i\to 0$ on $C_i\times \textbf{Q}_i$ and
the relative flag scheme
$$\sR_i=\underset{x\in I_i}{\times_{\textbf{Q}_i}}
Flag_{\vec n(x)}(\sF^i_x)\to \textbf{Q}_i.$$ Let
$\sF=\sF^1\oplus\sF^2$ denote direct sum of pullbacks of $\sF^1$,
$\sF^2$ on $$\wt C\times
(\textbf{Q}_1\times\textbf{Q}_2)=(C_1\times\textbf{Q}_1)\sqcup(C_2\times\textbf{Q}_2).$$
Let $\sE$ be the pullback of $\sF$ to $\wt
C\times(\sR_1\times\sR_2)$, and
$$\rho:\widetilde{\sR}=Grass_r(\sE_{x_1}\oplus\sE_{x_2})\to\sR=\sR_1\times\sR_2\to
\textbf{Q}=\textbf{Q}_1\times\textbf{Q}_2.$$
For the given $\omega=(\{k,\vec n(x),\,\vec a(x)\}_{x\in I_1\cup I_2},\mathcal{O}(1))$, let
$\wt{\sR}_{\omega}^{ss}$ (resp. $\wt{\sR}_{\omega}^{s}$) denote the open set of GIT semi-stable (resp. GIT stable) points under
action of $G=({\rm GL}(V_1)\times {\rm GL}(V_2))\cap {SL}(V_1\oplus V_2)$ on $\wt\sR$ respect to the polarization determined by $\omega$.
Let $\sH\subset\wt{\sR}$ be the open set defined in Notation \ref{nota3.10}, then for any data $\omega$ we have
$$\wt{\sR}_{\omega}^{s}\subset\wt{\sR}_{\omega}^{ss}\subset\sH.$$
The moduli space in Theorem \ref{thm3.18} is nothing but the GIT quotient
$$\psi:\wt{\sR}_{\omega}^{ss}\to \sP_{\omega}:=\wt{\sR}_{\omega}^{ss}//G.$$

There exists a morphism $\hat{\rm Det}_{\sH}: \sH\to J^d_{\wt C}=J^{d_1}_{C_1}\times J^{d_2}_{C_2}$, which extends
$$\hat{\rm Det}_{\sH_F}: \sH_F\to J^{d_1}_{C_1}\times J^{d_2}_{C_2},\quad (E,Q)\mapsto ({\rm det}(E|_{C_1}), {\rm det}(E|_{C_2}))$$
on the open set $\sH_F\subset \sH$ of GPB (i.e. GPS $(E,Q)$ with $E$ locally free) and induces a flat determinant morphism
$${\rm Det}_{\sP_{\omega}}:\sP_{\omega}\to J^d_{\wt C}=J^{d_1}_{C_1}\times J^{d_2}_{C_2}$$
(see page 46 of \cite{Su3} for detail). In fact, for any $L\in J^d_{\wt C}=J^{d_1}_{C_1}\times J^{d_2}_{C_2}$, let
\ga{3.9} {\sP_{\omega}^L:={\rm Det}_{\sP_{\omega}}^{-1}(L)\subset \sP_{\omega}}
and note that abelian variety $J^0_C=J^0_{C_1}\times J^0_{C_2}$ acts on $\sP_{\omega}$, the induced morphism
$\sP^L_{\omega}\times J^0_X\to \sP_{\omega}$ is a finite cover (see the proof of Lemma 6.6 in \cite{Su3}). Similarly, let $\sH^L=\hat{\rm Det}_{\sH}^{-1}(L)$ and $(\wt{\sR}_{\omega}^{ss})^L=\wt{\sR}_{\omega}^{ss}\cap \sH^L$, then
$$\psi:(\wt{\sR}_{\omega}^{ss})^L\to \sP^L_{\omega}=(\wt{\sR}_{\omega}^{ss})^L//G.$$

We do not have good estimate of ${\rm Codim}(\sH\setminus \wt{\sR}_{\omega}^{ss})$ since sub-sheaves $(E_1, Tor(E_2))$, $(Tor(E_1), E_2)$
of $E=(E_1,E_2)$ with rank $(r,0)$, $(0,r)$ may destroy semi-stability of $(E,Q)$ where $Tor(E_i)\subset E_i$ ($i=1,\,2$) are torsion sub-sheaves.
But we have estimate of ${\rm Codim}(\sH_{\omega}\setminus \wt{\sR}_{\omega}^{ss})$, where
$$\sH_{\omega}=\left\{\aligned&\text{$(E,Q)\in\sH$, with
$n_j^{\omega}\le \chi(E_j)=\chi_j\le n_j^{\omega}+r$ ($j=1,\,2$), and}\\&\text{${\rm dim}({\rm Tor}(E_1))\le n^{\omega}_2+r-\chi_2$, $\,\,{\rm dim}({\rm Tor}(E_2))\le n^{\omega}_1+r-\chi_1$}\endaligned\right\}.$$

\begin{prop}\label{prop3.19}
Let $\sD_1^f=\hat\sD_1\cup\hat\sD_1^t$
and $\sD_2^f=\hat\sD_2\cup\hat\sD_2^t$, where $\hat\sD_i\subset \wt{\sR}$ is the Zariski closure of
$\hat\sD_{F,\,i}\subset\wt{\sR}_F$ consisting of $(E,Q)\in\wt{\sR}_F$ that $E_{x_i}\to Q$ is not an isomorphism, and
${\hat\sD}_1^t\subset\wt{\sR}$ (rep. ${\hat\sD}_2^t\subset\wt{\sR}$) consists of $(E,Q)\in\wt{\sR}$ such that $E$ is not locally
free at $x_2$ (resp. at $x_1$). Then \begin{itemize}
\item [(1)] ${\rm Codim}(\sH_{\omega}^L\setminus(\wt\sR_{\omega}^{ss})^L)>
\underset{1\le i\le 2}{{\rm min}}\left\{(r-1)(g_i-\frac{r+3}{4})+\frac{|I_i|}{k}\right\};$
\item [(2)] ${\rm Codim}((\wt\sR_{\omega}^{ss})^L\setminus\{\sD_1^f\cup\sD^f_2\}\setminus(\wt\sR_{\omega}^{s})^L)>\underset{1\le i\le 2}{{\rm min}}\left\{(r-1)(g_i-1)+\frac{|I_i|}{k}\right\}$ when $n_1^{\omega}<\chi_1<n_1^{\omega}+r$;
\item [(3)] ${\rm Codim}((\wt\sR_{\omega}^{ss})^L\setminus\{\sD_1^f\cup\sD^f_2\}\setminus W_{\omega})\ge\underset{1\le i\le 2}{{\rm min}}\left\{(r-1)(g_i-1)+\frac{|I_i|}{k}\right\}$
when $\chi_1=n_1^{\omega}$ or $n_1^{\omega}+r$, where
$$W_{\omega}:=\left\{(E,Q)\in(\wt\sR_{\omega}^{ss})^L \bigg|\begin{aligned}
& \text{$\frac{par\chi(E')-dim(Q^{E'})}{r(E')}<\frac{par\chi(E)-dim(Q)}{r(E)}$}\\
&\text{$\forall$ $E'\subset E$ of rank $(r_1,r_2)\neq (0,r)$, $(r,0)$, $(0,0)$}\end{aligned}\right\};$$
\item [(4)] ${\rm Codim}((\wt\sR^{ss}_{\omega})^L\setminus W_{\omega})\ge \underset{1\le i\le 2}{{\rm min}}\left\{(r-1)(g_i-\frac{r+3}{4})+\frac{|I_i|}{k}\right\}$.
\end{itemize}
\end{prop}

\begin{proof} The statements (1), (2) and (3) are in fact reformulations of Proposition 6.3 in \cite{Su3} where determinants are not fixed. (4) follows the proof of
Proposition 6.3 in \cite{Su3} (see Remark 6.7 (2) of \cite{Su3}).
\end{proof}

\begin{prop}\label{prop3.20} For any $\omega$, let $\wt{\sR}_{\omega}^{ss}\xrightarrow{\psi} \sP_{\omega}:=\wt{\sR}_{\omega}^{ss}//G$ and assume
$$\underset{1\le i\le 2}{{\rm min}}\left\{(r-1)(g_i-\frac{r+3}{4})+\frac{|I_i|}{k}\right\}\ge 2.$$
Then $(\psi_*\omega_{\wt{\sR}_{\omega}^{ss}})^{inv.}=\omega_{\sP_{\omega}}$. For
$\omega_c=(2r,\{\vec n(x),\,\vec a_c(x)\}_{x\in I_1\cup I_2})$ satisfying
$$\underset{1\le i\le 2}{{\rm min}}\left\{(r-1)(g_i-\frac{r+3}{4})+\frac{|I_i|}{2r}\right\}\ge 2,$$
there is an ample line bundle $\Theta_{\sP_{\omega_c}}$ on $\sP_{\omega_c}$ such that
$$\omega^{-1}_{\sP_{\omega_c}}=\Theta_{\sP_{\omega_c}}\otimes {\rm Det}_{\sP_{\omega_c}}^*(\Theta_{J^d_{\wt C}}^{-1}).$$
In particular, for any $L\in J^d_{\wt C}$, $\sP^L_{\omega_c}$ is a normal Fano variety with only rational singularities.
\end{prop}

\begin{proof} According to a result of Knop in \cite{Kn} (see Lemma 4.17 of \cite{NR} for its global formulation), to prove $(\psi_*\omega_{\wt{\sR}_{\omega}^{ss}})^{inv.}=\omega_{\sP_{\omega}}$, it is enough to show that (1) the subset where the action of $G$ is not free has codimension at least two;
(2) for every prime divisor $D$ in $\wt{\sR}_{\omega}^{ss}$, $\psi(D)$ has codimension at most $1$.

To verify condition (1), when $n_1^{\omega}<\chi_1<n_1^{\omega}+r$, we have
$${\rm Codim}(\wt\sR_{\omega}^{ss}\setminus\{\sD_1^f\cup\sD^f_2\}\setminus\wt\sR_{\omega}^{s})>\underset{1\le i\le 2}{{\rm min}}\left\{(r-1)(g_i-1)+\frac{|I_i|}{k}\right\}\ge 2.$$
Note that $\sD^f_j=\hat\sD_j\cup \hat\sD^t_j$ where $\hat\sD_j$, $\hat\sD^t_j$ are irreducible, normal subvarieties (see Proposition C.7 of \cite{NR}), and the subsets of $\hat\sD_j$ and $\hat\sD^t_j$, where the action of $G$ is free, are open subsets.
Thus it is enough to find a $(E, Q)\in\sD^f_j$ ($j=1,\,2$) such that its automorphisms are only scales.
Let $E_i'$ ($i=1,\,2$) be stable parabolic bundles of rank $r$ and $\chi(E'_i)=\chi_i'$ on $C_i$ with parabolic structures determined by $(k, \{\vec n(x),\vec a(x)\}_{x\in I_i})$. If we take $\chi_1'=\chi_1-1$, $\chi_2'=\chi_2$, let $E_1=E_1'\oplus\,_{x_1}\mathbb{C}\cdot\beta$, $E_2=E_2'$ and
$E=(E_1,E_2)$, the surjection $E_{x_1}\oplus E_{x_2}\xrightarrow{q} Q$ is defined by any isomorphism $E_{x_2}\xrightarrow{q_2} Q$ and
a linear map $E_{x_1}=(E'_1)_{x_1}\oplus \mathbb{C}\cdot\beta\xrightarrow{q_1} Q$ such that $q_1(\beta)\neq 0$ and $q_1|_{(E'_1)_{x_1}}\neq 0$. Then $(E,Q)\in {\hat\sD}_2^t$ by definition.
To see ${\rm Aut}((E,Q))=\mathbb{C}^*$, let $0\to K\to E_{x_1}\oplus E_{x_2}\xrightarrow{q}Q\to 0$ and $(E,Q)\xrightarrow{\Phi}(E,Q)$ be an isomorphism. Then $E\xrightarrow{\Phi}E$ is an isomorphism of parabolic bundles such that $\Phi_{x_1+x_2}(K)=K$. Since $E_1'$, $E_2$ are stable, $\Phi|_{E_1}=(\lambda'_1,\lambda_1):E_1'\oplus\,_{x_1}\mathbb{C}\beta\to E_1'\oplus\,_{x_1}\mathbb{C}\beta$ and $\Phi|_{E_2}=\lambda_2: E_2\to E_2$ for nonzero constants
$\lambda_1',\,\lambda_1,\,\lambda_2$. The requirement $\Phi_{x_1+x_2}(K)=K$ implies that $\lambda'_1=\lambda_1=\lambda_2$.
In fact, $K=\{\,(\alpha, f(\alpha))\in E_{x_1}\oplus E_{x_2}\,|\,\forall\,\alpha\in E_{x_1}\}$ where $f=-q_2^{-1}q_1: E_{x_1}\to E_{x_2}$. For any $\alpha=\alpha'+\beta\in E_{x_1}$, we have
$$\Phi_{x_1+x_2}(\alpha, f(\alpha))=(\lambda'_1\alpha'+\lambda_1\beta, \lambda_2f(\alpha')+\lambda_2f(\beta))\in K$$
which implies that $\lambda_2f(\alpha')+\lambda_2f(\beta)=\lambda'_1f(\alpha')+\lambda_1f(\beta)$. Thus $\lambda_1=\lambda_2$ (by taking $\alpha'=0$)
and $\lambda_1'=\lambda_2$ (by taking $\alpha'$ such that $q_1(\alpha')\neq 0$).
Similarly, one can find such $(E,Q)\in  {\hat\sD}_1^t$. To construct $(E, Q)\in \hat\sD_j$ with ${\rm Aut}(E,Q)=\mathbb{C}^*$, we take $\chi'_i=\chi_i$, $E_i=E_i'$ and
$E=(E_1, E_2)$ with $$E_{x_1}\oplus E_{x_2}\xrightarrow{q}Q\to 0$$
defined by any isomorphism $q_2: E_{x_2}\to Q$ and nontrivial linear map $q_1:E_{x_1}\to Q$ (which is not surjective). Thus $(E,Q)\in \hat\sD_1$ and ${\rm Aut}(E,Q)=\mathbb{C}^*$. Similarly one can find such $(E,Q)\in \hat\sD_2$. When $\chi_1=n_1^{\omega}$ or $n_1^{\omega}+r$, we have
${\rm Codim}(\wt\sR_{\omega}^{ss}\setminus\{\sD_1^f\cup\sD^f_2\}\setminus W_{\omega})\ge 2$. Thus we only need to show, for any
$(E,Q)\in (\wt\sR_{\omega}^{ss}\setminus\{\sD_1^f\cup\sD^f_2\})\cap W_{\omega}$, ${\rm Aut}((E,Q))=\mathbb{C}^*$. This is easy since the proof of
Lemma 6.1 (4) in \cite{Su3} implies stability of parabolic bundles $E_1$ and $E_2$. Thus any automorphism of $E=(E_1,E_2)$ must be of type
$(\lambda_1id_{E_1},\lambda_2id_{E_2})$, which induces an automorphism of $(E,Q)$ if and only if $\lambda_1=\lambda_2$.

To verify condition (2), if a prime divisor $D$ is not contained in $\wt\sR_{\omega}^{ss}\setminus W_{\omega}$, $\psi(D)$ is a divisor. If $D$ is contained
in $\wt\sR_{\omega}^{ss}\setminus W_{\omega}$, then $D$ must be one of $\hat\sD_j$, $\hat\sD^t_j$ since ${\rm Codim}(\wt\sR_{\omega}^{ss}\setminus\{\sD_1^f\cup\sD^f_2\}\setminus W_{\omega})\ge 2$. However, $\psi(\hat\sD_j)=\psi(\hat\sD^t_j)=\sD_j$ ($j=1,\,2$)
by Proposition 2.5 of \cite{Su2}, which are divisors. Thus we have proved that $(\psi_*\omega_{\wt{\sR}_{\omega}^{ss}})^{inv.}=\omega_{\sP_{\omega}}$.

When $\omega=\omega_c$, by Proposition 6.4 of \cite{Su3}, there is an ample line bundle $\Theta_{\sP_{\omega_c}}$ on $\sP_{\omega_c}$ such that
$\omega_{\wt{\sR}_{\omega_c}^{ss}}^{-1}=\psi^*(\Theta_{\sP_{\omega_c}}\otimes {\rm Det}_{\sP_{\omega_c}}^*(\Theta_{J^d_{\wt C}}^{-1})).$ Thus
$$\omega_{\sP_{\omega_c}}=(\psi_*(\omega_{\wt{\sR}_{\omega_c}^{ss}}))^{inv.}=\Theta_{\sP_{\omega_c}}^{-1}\otimes {\rm Det}_{\sP_{\omega_c}}^*(\Theta_{J^d_{\wt C}})$$
and $\omega^{-1}_{\sP^L_{\omega_c}}=\Theta_{\sP_{\omega_c}}|_{\sP^L_{\omega_c}}$ is ample, $\sP^L_{\omega_c}$ is a normal Fano variety with only rational singularities.
\end{proof}

\begin{lem}\label{lem3.21} Let $V$ be a normal variety acting by a reductive group $G$. Suppose a good quotient $\phi:V\to U$ exists. Let $\sL$ be a line bundle on $U$ and
$\wt \sL=\phi^*(\sL)$. Let $V''\subset V'\subset V$ be open $G$-invariant subvarieties of $V$ such that $\phi(V')=U$ and $V''=\phi^{-1}(U'')$ for some nonempty open subset
$U''\subset U$. Then $\phi^G_*(\wt\sL|_{V'})=\sL$ (i.e, for any nonempty open set $X\subset U$,
${\rm H}^0(\phi^{-1}(X),\wt\sL)^{inv.}\to {\rm H}^0(V'\cap\phi^{-1}(X),\wt\sL)^{inv.}$ is an isomorphism).

\end{lem}

\begin{proof} It is in fact a reformulation of Lemma 4.16 in \cite{NR}, where
$${\rm H}^0(V,\wt\sL)^{inv.}\to {\rm H}^0(V',\wt\sL)^{inv.}$$ was shown to be an isomorphism.
\end{proof}

\begin{thm}\label{thm3.22} For any data $\omega=(\{k,\vec n(x),\,\,\vec a(x)\}_{x\in I_1\cup I_2},\sO(1))$ and integers $\chi_1$, $\chi_2$ satisfying $\chi_1+\chi_2=\chi+r$, $n_j^{\omega}\le\chi_j\le n_j^{\omega}+r$ ($j=1,\,2$), let $\sP^L_{\omega}$ be
the coarse  moduli space  of $s$-equivalence classes of semi-stable GPS $E=(E_1, E_2)$ on $\wt C$ with fixed determinant $L$, $\chi(E_j)=\chi_j$ and parabolic structures at the points of $I$ given by the data $\omega$. Then $\sP^L_{\omega}$ is of globally $F$-regular type.
\end{thm}

\begin{proof} Let $I'_i\subset X_i\setminus (I_i\cup\{x_i\})$ be a subset and $I'=I'_1\cup I'_2$. Recall
$$\sR_i=\underset{x\in I_i}{\times_{\textbf{Q}_i}}
Flag_{\vec n(x)}(\sF^i_x)\to \textbf{Q}_i$$ and
$\rho:\widetilde{\sR}=Grass_r(\sF^1_{x_1}\oplus\sF^2_{x_2})\to\sR=\sR_1\times\sR_2$, let
$$\sR_i'=\underset{x\in I_i\cup I_i'}{\times_{\mathbf{Q}}}
Flag_{\vec n(x)}(\sF^i_x)\to \sR_i,\quad \sR'=\sR_1'\times\sR_2'\xrightarrow{\hat f}\sR=\sR_1\times \sR_2$$ be the projection and
$\wt\sR':=\wt\sR\times_{\sR}\sR'\xrightarrow{\hat f}\wt\sR$ be induced via the diagram
$$\CD
  {\wt\sR}' @>\rho>> {\sR}' \\
  @V \hat f VV @V \hat f VV  \\
  \wt\sR @>\rho>> \sR
\endCD$$
Then, on $\sH^L\subset \wt\sR$,  it is clear that $(\sH')^L:={\hat f}^{-1}(\sH^L)\xrightarrow{\hat f} \sH^L$ is a $G$-invariant and
$p$-compatible morphism such that $\hat f_*\sO_{(\sH')^L}=\sO_{\sH^L}$.

For $\omega=(k, \{\vec n(x),\,\vec a(x)\}_{x\in I},\sO(1))$, $\omega_c=(2r, \{\vec n(x),\,\vec a_c(x)\}_{x\in I\cup I'},\sO(1))$, we have
$(\wt\sR_{\omega}^{ss})^L\subset\sH^L$, $\,(\wt\sR_{\omega_c}^{\prime\,ss})^L\subset (\sH')^L$. Moreover, for $\omega_c$,
let $\ell_j^c=2\chi_j-r-\sum_{x\in I_j\cup I'_j}r_{l_x}(x)$ and $\ell^c=\ell^c_1+\ell^c_2=2\chi-\sum_{x\in I\cup I'}r_{l_x}(x).$
Then $$\sum_{x\in I\cup I'}\sum^{l_x}_{i=1} (\bar a_{i+1}(x)-\bar a_i(x))r_i(x) +r\ell^c= 2r\chi.$$
The choices of $\{\vec n(x)\}_{x\in I'}$ satisfying $\ell^c_j=\frac{c_j}{c_1+c_2}\ell^c$ for arbitrary large $|I'_1|$ and $|I'_2|$ are possible and it is easy to compute that $n_j^{\omega^c}=\chi_j-\frac{r}{2}$, thus
$$n_j^{\omega^c}<\chi_j<n_j^{\omega^c}+r\qquad (j=1,\,\,2).$$
Recall $(\wt\sR_{\omega}^{ss})^L\xrightarrow{\psi}\sP^L_{\omega}:=Y,\quad (\wt\sR_{\omega_c}^{\prime\,ss})^L\xrightarrow{\varphi}\sP^L_{\omega_c}:=Z$, choose $I_i'$ satisfying
$$\underset{1\le i\le 2}{{\rm min}}\left\{(r-1)(g_i-\frac{r+3}{4})+\frac{|I_i|+|I_i'|}{2r}\right\}\ge 2.$$
Then $Z$ is a normal Fano variety with only rational singularities by Proposition \ref{prop3.20}, which is in particular of globally $F$-regular type.
To apply Proposition \ref{prop2.10}, let $$W=W_{\omega_c}\subset (\wt\sR_{\omega_c}^{\prime\,ss})^L,\quad \hat X=W\cap\hat f^{-1}((\wt\sR_{\omega}^{ss})^L).$$
For any $(E, Q)\in (\wt\sR_{\omega_c}^{\prime\,ss})^L\setminus(\wt\sR_{\omega_c}^{\prime\,s})^L$, there is an exact sequence
\ga{3.10}{0\to (E',Q')\to (E,Q)\to (\wt E,\wt Q)\to 0}
in the category $\sC_{\mu}$ (see Proposition 2.4 of \cite{Su2}) such that $(\wt E, \wt Q)$ is stable (respect to $\omega_c$). Then either $\wt E$ is torsion free when $r(\wt E)>0$ or $(\wt E,\wt Q)=(\,_{x_i}\mathbb{C}, \mathbb{C})$. If $(E,Q)\in W$, $\wt E$ has rank $r$ or rank $(r, 0)$, $(0,r)$ when $r(\wt E)>0$. Thus
it is easy to show that $(E,Q)\in W$ if and only if $gr(E,Q)$ is one of the following\begin{itemize}
\item [(1)] $gr(E,Q)=(\wt E, \wt Q)\oplus (\,_{x_1}\tau_1\oplus\,_{x_2}\tau_2, \tau_1\oplus\tau_2)$ where $(\wt E,\wt Q)\in\sC_{\mu}$ is stable of rank $(r,r)$;
\item [(2)] $gr(E,Q)=(\wt E_1,\wt Q_1)\oplus (\wt E_2,\wt Q_2)\oplus (\,_{x_1}\tau_1\oplus\,_{x_2}\tau_2, \tau_1\oplus\tau_2)$
where $(\wt E_1,\wt Q_1)$ and $(\wt E_2,\wt Q_2)\in \sC_{\mu}$ are stable of rank $(r,0)$ and $(0,r)$,\end{itemize}
which implies that $\varphi^{-1}\varphi(W)=W$. Hence, to check that $\hat X=\varphi^{-1}\varphi(\hat X)$, it is enough to show that $(E,Q)$
is semi-stable (respect to $\omega$) if and only if the above GPS
$(\wt E,\wt Q)$, $(\wt E_1,\wt Q_1)$ and $(\wt E_2,\wt Q_2)$ in (1) and (2) are semi-stable (respect to $\omega$) with the same slope $\mu_{\omega}(E,Q)$, which is easy to
check by using \eqref{3.10} when either $\wt E$ has rank $r$ or $r(\wt E)=0$. If $\wt E$ has rank $(0,r)$, $E'$ must have rank $(r,0)$. Then $(E,Q)$ is $\omega$-semistable if and only if $(E',Q')$, $(\wt E,\wt Q)$ are $\omega$-semistable with $\mu_{\omega}(E',Q')=\mu_{\omega}(\wt E,\wt Q)=\mu_{\omega}(E,Q)$ since the exact sequence
\eqref{3.10} is split in this case.

Now $\hat X \xrightarrow{\varphi} X:=\varphi(\hat X)\subset Z$ is a category quotient and the $G$-invariant $(\sH')^L\xrightarrow{\hat f} \sH^L$ induces a morphism $f: X \to Y$ such that
$$\CD
 (\sH')^L\supset\hat X @>\varphi>> X \\
  @V \hat f VV @V f VV  \\
  \sH^L\supset(\wt\sR_{\omega}^{ss})^L @>\psi>> Y
\endCD$$ is a commutative diagram. However we do not have $${\rm Codim}((\sH')^L\setminus W)\ge 2$$ as required in Proposition \ref{prop2.10}, which was used to prove
$f_*\sO_X=\sO_Y$.

On the other hand, let $(\wt\sR_{\omega, F}^{ss})^L\subset (\wt\sR_{\omega}^{ss})^L$, $(\sH_F')^L \subset (\sH')^L$ be the $G$-invariant open sets of GPS $(E,Q)$ with $E$ being locally free. Then
$${\rm Codim}((\sH_F')^L\setminus W)\ge\underset{1\le i\le 2}{{\rm min}}\left\{(r-1)(g_i-\frac{r+3}{4})+\frac{|I_i|+|I_i'|}{k}\right\}\ge 2$$
by Proposition \ref{prop3.19} (1) and (2), which and Lemma \ref{lem3.21} imply
$$f_*\sO_X=\sO_Y.$$
In fact, let $\hat X_F=\hat X\cap (\sH_F')^L$ and apply Lemma \ref{lem3.21} to the surjections
$$\hat X_F\xrightarrow\varphi X, \quad (\wt\sR_{\omega, F}^{ss})^L\xrightarrow\psi Y,$$
we have $\varphi^G_*\sO_{\hat X_F}=\sO_X$, $\psi^G_*\sO_{(\wt\sR_{\omega, F}^{ss})^L}=\sO_Y$, which means $\forall \,\,U\subset Y$,
$$\sO_Y(U)=H^0(\psi^{-1}(U),\sO_{(\wt\sR_{\omega,F}^{ss})^L})^{inv.}$$ $$\sO_X(f^{-1}(U))=H^0(\varphi^{-1}(f^{-1}(U)),\sO_{\hat X_F})^{inv.}.$$
On the other hand, for the $G$-invariant morphism $(\sH')^L\xrightarrow{\hat f}\sH^L$ with $\hat f_*\sO_{(\sH')^L}=\sO_{\sH^L}$,
by using the fact that $$\hat f^{-1}\psi^{-1}(U)\setminus\hat f^{-1}\psi^{-1}(U)\cap\hat X=\hat f^{-1}\psi^{-1}(U)\cap ((\sH_F')^L\setminus W)$$
has at least codimension two, we have
$$\aligned\sO_Y(U)&=H^0(\psi^{-1}(U),\sO_{(\wt\sR_{\omega,F}^{ss})^L})^{inv.}=H^0(\hat f^{-1}\psi^{-1}(U),\sO_{(\sH_F')^L})^{inv.}\\&=
H^0(\hat f^{-1}\psi^{-1}(U)\cap\hat X,\sO_{(\sH_F')^L})^{inv.}\\&=H^0(\varphi^{-1}(f^{-1}(U)),\sO_{\hat X_F})^{inv.}=\sO_X(f^{-1}(U)).
\endaligned$$
Thus $Y=\sP^L_{\omega}$ is of globally $F$-regular type since $X$ is so.
\end{proof}

\section{Vanishing theorems and recurrence relations}

In this section, we use the main results of Section 3 to prove vanishing theorems on moduli spaces of parabolic sheaves on curves with at most one node, and to establish
recurrence relations of the dimension of generalized theta functions. As an immediate application of globally $F$-regular type of the moduli spaces of parabolic sheaves on a
smooth projective curve $C$, we have the following vanishing theorem

\begin{thm}\label{thm4.1} Let $\sU_{C,\,\omega}$ be the moduli
space of semistable parabolic bundles
of rank $r$ and degree $d$ on a smooth projective curve $C$ with parabolic structures determined by $\omega=(k,\{\vec n(x),\vec a(x)\}_{x\in I})$.
Then $${\rm H}^i(\sU_{C,\,\omega},\sL)=0 \quad \forall\,\,i>0$$
for any ample line bundle $\sL$ on $\sU_{C,\,\omega}$.
\end{thm}

\begin{proof} Let ${\rm Det}: \sU_{C,\,\omega}\to J^d_C$ and $\sU_{C,\,\omega}^L={\rm Det}^{-1}(L)$, then morphism
$$J^0_C\times \sU_{C,\,\omega}^L\to \sU_{C,\,\omega},\quad (L_0, E)\mapsto L_0\otimes E$$
is a $r^{2g}$-fold cover. Then it is enough to show $H^i(J^0_C\times \sU_{C,\,\omega}^L, \sL)=0$ for $i>0$ and any ample
line bundle $\sL$. Since $\sU_{C,\,\omega}^L$ is of globally $F$-regular type, $H^1(\sU_{C,\,\omega}^L, \sO_{\sU_{C,\,\omega}^L})=0$ by Theorem \ref{thm2.5} (2).
Thus $\sL=\sL_1\otimes\sL_2$ where $\sL_1$ (resp. $\sL_2$) is an ample bundle on $J^0_C$ (resp. $\sU_{C,\,\omega}^L$) and
$$H^i(J^0_C\times \sU_{C,\,\omega}^L, \sL)=H^i(J^0_C,\sL_1)\otimes H^0( \sU_{C,\,\omega}^L,\sL_2)=0.$$
\end{proof}

For any irreducible curve $C$ with at most one node $x_0\in C$, there is an algebraic family of ample line bundles $\Theta_{\sU_{C,\,\omega}}$ on
$\sU_{C,\,\omega}$ when
\ga{4.1}{\text{$\ell:=\frac{k\chi-\sum_{x\in
I}\sum^{l_x}_{i=1}d_i(x)r_i(x)}{r}$ is an integer}}
(see Theorem 3.1 of \cite{Su3}). Then Theorem \ref{thm4.1} implies that the number
\ga{4.2} {D_g(r,d,\omega)=dim H^0(\sU_{C,\,\omega},\Theta_{\sU_{C,\,\omega}})}
is independent of $C$, parabolic points $x\in I$ (of course, depending on the number $|I|$ of parabolic points) and the choice of $\Theta_{\sU_{C,\,\omega}}$ in the algebraic family when $C$ is smooth.

When $C$ has one node $x_0\in C$, the moduli spaces $\sU_{C,\,\omega}$ are only seminormal (see Theorem 4.2 of \cite{Su1}) and its normalization
$$\phi:\sP_{\omega}\to \sU_{C,\,\omega}$$
is the coarse  moduli space $\sP_{\omega}$ of $s$-equivalence classes of semi-stable GPS on $\wt C\xrightarrow{\pi} C$ with generalized parabolic structures on $\pi^{-1}(x_0)=x_1+x_2$ and parabolic structures at the points of $\pi^{-1}(I)$ given by the data $\omega$ (see Proposition 2.1 of \cite{Su1}, or Proposition 3.1
of \cite{Su2}).

\begin{lem}\label{lem4.2}[Lemma 5.5 of \cite{Su1}] For any line bundle $\sL$ on $\sU_{C,\,\omega}$,
$$\phi^*:H^1(\sU_{C,\,\omega},\sL)\to H^1(\sP_{\omega},\phi^*\sL)$$
is injective.
\end{lem}

\begin{thm}\label{thm4.3} Let $\Theta_{\sP_{\omega}}=\phi^*\Theta_{\sU_{C,\,\omega}}$. Then
$H^i(\sP_{\omega},\Theta_{\sP_{\omega}})=0$ for any $i>0$. In particular, $H^1(\sU_{C,\,\omega}, \Theta_{\sU_{C,\,\omega}})=0.$
\end{thm}

\begin{proof} Let ${\rm Det}:\sP_{\omega}\to J^d_{\wt C}$ be the flat morphism defined in \eqref{3.5}. Then $H^i(\sP_{\omega},\Theta_{\sP_{\omega}})=0$
follows the facts that $R^i{\rm Det}_*\Theta_{\sP_{\omega}}=0$ by Theorem \ref{thm3.14} and $H^i(J^d_{\wt C},{\rm Det}_*\Theta_{\sP_{\omega}})=0$ by a decomposition of
${\rm Det}_*\Theta_{\sP_{\omega}}$ (see Remark 4.2 of \cite{Su1} or a more precise version in Lemma 5.2 of \cite{Su3}).
\end{proof}

When $C=C_1\cup C_2$ is a reducible one nodal curve, we have a stronger vanishing theorem on $\sU_{C,\,\omega}$ and $\sP_{\omega}$.

\begin{thm}\label{thm4.4} When $C$ is a reducible one nodal curve with two smooth irreducible components, let $\sP_{\omega}$ be the moduli spaces of semi-stable
GPS on $\wt C$ with parabolic structures determined by $\omega$. Then, for any ample line bundle $\wt\sL$ on $\sP_{\omega}$ and $i>0$, we have $H^i(\sP_{\omega},\wt\sL)=0.$ In particular,
$$H^1(\sU_{C,\,\omega},\, \sL)=0$$
holds for any ample line bundle $\sL$ on $\sU_{C,\,\omega}$.
\end{thm}

\begin{proof} By Lemma \ref{lem4.2}, it is enough to show $H^i(\sP_{\omega},\wt\sL)=0$ for any ample line bundle $\wt\sL$ and $i>0$.

When $C=C_1\cup C_2$, the moduli space
$\sP_{\omega}$ is a disjoint union of
$$\{\sP_{d_1,d_2}\}_{d_1+d_2=d}$$ where $\sP_{d_1,d_2}$ consists of GPS $(E,Q)$ with $d_i=deg(E|_{C_i})$. It is enough to consider
$\sP_{\omega}=\sP_{d_1,d_2}$, thus we have the flat morphism
$${\rm Det}:\sP_{\omega}\to J^d_{\wt C}=J^{d_1}_{C_1}\times
J^{d_2}_{C_2}=J^d_C$$ and $J^0_{\wt C}=J^0_{C_1}\times
J^0_{C_2}=J^0_C$ acts on $\sP_{\omega}$ by
$$((E,Q),\sN)\mapsto (E\otimes\pi^*\sN, Q\otimes \sN_{x_0})$$
where $\pi:\wt C\to C$ is the normalization of $C$.
Let $\sP_{\omega}^L={\rm Det}^{-1}(L)$ and consider the morphism
$f:\sP_{\omega}^L\times J^0_C\to \sP_{\omega}$, which is a finite morphism (see the proof of Lemma 6.6 in \cite{Su3} where we figure out a line bundle
$\Theta$ on $\sP_{\omega}$ such that its pullback $f^*(\Theta)$ is ample). Thus it is enough to prove the vanishing theorem on $\sP_{\omega}^L\times J^0_C$,
which follows the same arguments in the proof of Theorem \ref{thm4.1} by using Theorem \ref{thm3.22}.
\end{proof}

\begin{nota}\label{nota4.5} For $\mu=(\mu_1,\cdots,\mu_r)$  with $0\le\mu_r\le\cdots\le\mu_1<
k,$ let $$\{d_i=\mu_{r_i}-\mu_{r_i+1}\}_{1\le i\le l}$$ be the
subset of nonzero integers in
$\{\mu_i-\mu_{i+1}\}_{i=1,\cdots,r-1}.$ We define
$$ r_i(x_1)=r_i,\quad d_i(x_1)=d_i,\quad l_{x_1}=l,$$
$$ r_i(x_2)=r-r_{l-i+1},\quad d_i(x_2)=d_{l-i+1},\quad l_{x_2}=l,$$ and for $j=1,2$, we set
$$\aligned
\vec a(x_j)&=\left(\mu_r,\mu_r+d_1(x_j),\cdots,\mu_r+
\sum^{l_{x_j}-1}_{i=1}d_i(x_j),\mu_r+\sum^{l_{x_j}}_{i=1}d_i(x_j)\right)\\
\vec n(x_j)&=(r_1(x_j),r_2(x_j)-r_1(x_j),
\cdots,r_{l_{x_j}}(x_j)-r_{l_{x_j}-1}(x_j),r-r_{l_{x_j}}(x_j)).\endaligned$$
\end{nota}

\begin{thm}\label{thm4.6} For any $\omega=(k,\,\{\vec n(x),\vec a(x)\}_{x\in I})$ such that
$$\text{$\ell:=\frac{k\chi-\sum_{x\in
I}\sum^{l_x}_{i=1}d_i(x)r_i(x)}{r}$ is an integer}$$
where $\chi=d+r(1-g)$, let $D_g(r,d,\omega)=dim H^0(\sU_{C,\,\omega},\Theta_{\sU_{C,\,\omega}})$. Then, for any positive integers $c_1$, $c_2$ and partitions $I=I_1\cup I_2$, $g=g_1+g_2$
such that $\ell_j=\frac{c_j\ell}{c_1+c_2}$ ($j=,\,2$) are integers, we have
\ga{4.3}{D_g(r,d,\omega)=\sum_{\mu}D_{g-1}(r,d,\omega^{\mu})}
\ga{4.4}{D_g(r,d,\omega)=\sum_{\mu}D_{g_1}(r,d_1^{\mu},\omega_1^{\mu})\cdot D_{g_2}(r,d_2^{\mu},\omega_2^{\mu})}
where $\mu=(\mu_1,\cdots,\mu_r)$ runs through $0\le\mu_r\le\cdots\le
\mu_1< k$ and $$\omega^{\mu}=(k, \{\vec n(x),\,\vec a(x)\}_{x\in I\cup\{x_1,\,x_2\}}),\,\,\omega_j^{\mu}=(k, \{\vec n(x),\,\vec a(x)\}_{x\in I_j\cup\{x_j\}})$$
with $\vec n(x_j)$, $\vec a(x_j)$ ($j=1,\,2$) determined by $\mu$ (Notation \ref{nota4.5}) and
$$d_1^{\mu}=n^{\omega}_1+\frac{1}{k}\sum^r_{i=1}
\mu_i+r(g_1-1),\quad d^{\mu}_2=n^{\omega}_2+r-\frac{1}{k}\sum^r_{i=1}
\mu_i+r(g_2-1)$$
$$n^{\omega}_j=\frac{1}{k}\left(r\frac{c_j}{c_1+c_2}\ell+\sum_{x\in
I_j}\sum^{l_x}_{i=1}d_i(x)r_i(x)\right)\,\,\,(j=1,\,\,2).$$
\end{thm}

\begin{proof} Consider a flat family of projective $|I|$-pointed curves $\sX\to T$ and a relative ample line bundle $\sO_{\sX}(1)$ of relative degree $c$ such that a fiber $\sX_{t_0}:=X$ ($t_0\in T$) is a connected curve with only one node $x_0\in X$ and $\sX_t$ ($t\in T\setminus\{t_0\}$) are smooth curves with a fiber $\sX_{t_1}=C$ ($t_1\neq t_0$). Then one can associate a family of moduli spaces $\sM\to T$ and a line bundle
$\Theta$ on $\sM$ such that each fiber $\sM_t=\sU_{\sX_t,\,\omega}$ is the moduli space of semi-stable parabolic sheaves on $\sX_t$ and $\Theta|_{\sM_t}=\Theta_{\sU_{\sX_t},\,\omega}$. By degenerating $C$ to an irreducible $X$ and using Theorem \ref{thm4.1} and Theorem \ref{thm4.3}, the recurrence relation \eqref{4.3} is nothing but the Factorization theorem of \cite{Su1}.
If we degenerate $C$ to a reducible curve $X=X_1\cup X_2$ with $g(X_i)=g_i$ and choose the relative ample line bundle $\sO_{\sX}(1)$ such that $c_i={\rm deg}(\sO_{\sX}(1)|_{X_i}$, by using
Theorem \ref{thm4.1} and Theorem \ref{thm4.4}, the recurrence relation \eqref{4.4} is exactly the Factorization theorem of \cite{Su2}.
\end{proof}

In the recurrence relation \eqref{4.4}, the degree $d_1^{\mu}$ varies with $\mu$ and $D_{g_1}(r,d_1^{\mu},\omega_1^{\mu})$ makes sense only when $d_1^{\mu}$ is an integer, which are not
convenient for applications. To remedy it, we are going to study the behavior of $D_g(r,d,\omega)$ under Hecke transformation.

Given a parabolic sheaf $E$ with quasi-parabolic structure
$$E_z=Q_{l_z+1}(E)_z\twoheadrightarrow Q_{l_z}(E)_z\twoheadrightarrow \cdots \cdots \twoheadrightarrow Q_1(E)_z\twoheadrightarrow Q_0(E)_z=0$$
of type $\vec n(z)=(n_1(z),...,n_{l_z+1}(z))$ at $z\in I$ and weights $$ 0= a_1(z) <a_2(z)< \cdots <a_{l_z+1}(z)<k.$$
Let $F_i(E)_z=ker\{E_z\twoheadrightarrow Q_i(E)_z\}$ and  $E'=ker\{E\twoheadrightarrow Q_1(E)_z\}$.
Then, at $z\in I$, $E'$ has a natural quasi-parabolic structure
\ga{4.5} {E'_z\twoheadrightarrow F_1(E)_z\twoheadrightarrow Q_{l_z-1}(E')_z \twoheadrightarrow\cdots\twoheadrightarrow Q_1(E')_z\twoheadrightarrow 0}
of type $\vec n'(z)=(n'_1(z),...,n'_{l_z+1}(z))=(n_2(z),...,n_{l_z+1}(z),n_1(z))$, where
$$Q_i(E')_z\subset Q_{i+1}(E)_z$$ is the image of
$F_1(E)_z$ under $E_z\twoheadrightarrow Q_{i+1}(E)_z$. It is easy to see
$$Q_i(E')_z\cong F_1(E)_z/F_{i+1}(E)_z.$$

\begin{defn}\label{defn4.7} The parabolic sheaf $E'$ with given weight
$$0= a'_1(z)<\cdots < a'_{l_z+1}(z)< k$$
is called Hecke transformation of the parabolic sheaf $E$ at $z\in I$, where $a'_{l_z+1}(z)=k-a_2(z)$
and $a'_i(z)=a_{i+1}(z)-a_2(z)$ for $2\leq i\leq l_z$.
\end{defn}

\begin{lem}\label{lem4.8}  The parabolic bundle $E'$ is semistable (resp., stable) iff $E$ is semistable (resp., stable).
\end{lem}
\begin{proof} $E'$ is defined by the exact sequence of sheaves
$$0\to E'\xrightarrow{i} E\xrightarrow{\delta} Q_1(E)_z\to 0$$
such that $E_z\xrightarrow{\delta_z} Q_1(E)_z$ is the surjective homomorphism
$$E_z=Q_{l_z+1}(E)_z\twoheadrightarrow Q_{l_z}(E)_z\twoheadrightarrow \cdots \cdots \twoheadrightarrow Q_1(E)_z.$$
For any sub-bundle $\sF\subset E$ of rank $r_1$, let $Q_i(E)^{\sF}_z\subset Q_i(E)_z$ be the image of $\sF_z\subset E_z$ under
$E_z\twoheadrightarrow Q_{l_z}(E)_z\twoheadrightarrow \cdots \cdots \twoheadrightarrow Q_i(E)_z$, and the sub-bundle $\sF'\subset E'$ is defined by exact sequence
of sheaves:
$$0\to \sF'\xrightarrow{i} \sF\xrightarrow{\delta} Q_1(E)^{\sF}_z\to 0.$$
Let $Q_{l_z}(E')_z^{\sF'}\subset F_1(E)_z$ and $Q_i(E')_z^{\sF'}\subset Q_i(E')_z$ ($1\le i<l_z$) be the image of $\sF'_z\subset E'_z$
under $E'_z\twoheadrightarrow F_1(E)_z$ and $E'_z\twoheadrightarrow Q_i(E')_z$, which are the surjections in \eqref{4.5}.
Since $Q_{l_z}(E')_z^{\sF'}=ker\{\sF_z\xrightarrow{\delta_z} Q_1(E)_z^{\sF}\}$,
$$ker\{Q_{i}(E')_z^{\sF'}\twoheadrightarrow Q_{i-1}(E')_z^{\sF'}\}=ker\{Q_{i+1}(E)_z^{\sF}\twoheadrightarrow Q_{i}(E)_z^{\sF}\}$$
for $1\le i\le l_z$. In particular, $n_i^{\sF'}(z)=n^{\sF}_{i+1}(z)$,
$n^{\sF'}_{l_z+1}(z)=n^{\sF}_1(z)$,
$${\rm p}ar_{\omega'}\chi(\sF')={\rm p}ar_{\omega}\chi(\sF)-\frac{r_1}{k}a_2(z), \quad {\rm p}ar_{\omega'}\chi(E')={\rm p}ar_{\omega}\chi(E)-\frac{r}{k}a_2(z).$$
Thus ${\rm p}ar_{\omega'}\mu(\sF')-{\rm p}ar_{\omega'}\mu(E')={\rm p}ar_{\omega}\mu(\sF)-{\rm p}ar_{\omega}\mu(E)$, which proves the lemma.
\end{proof}

\begin{lem}\label{lem4.9} For parabolic data $\omega=(k,\{\vec n(x),\vec a(x)\}_{x\in I})$, let
\ga{4.6} {\omega'=(k, \{\vec n(x), \vec a(x)\}_{x\neq z \in I}\cup \{\vec a'(z), \vec n'(z)\})}
where $\vec n'(z)==(n'_1(z),...,n'_{l_z+1}(z))=(n_2(z),...,n_{l_z+1}(z),n_1(z))$,
$$\vec a'(z)=(0, a_2'(z), ..., a'_{l_z+1}(z)),\quad a'_{l_z+1}(z)=k-a_2(z)+a_1(z)$$
and $a'_i(z)=a_{i+1}(z)-a_2(z)+a_1(z)$ for $2\le i\le l_z$. Then
$$D_g(r,d,\omega)=D_g(r,d-n_1(z),\omega').$$
\end{lem}

\begin{proof} One can also define the Hecke transformation of a family of parabolic sheaves (flat family yielding flat family, and preserve semistability). Thus, for $z\in I$, we have a morphism $$\mathrm{H}_z: \sU_{C,\,\omega}=\sU_C(r, d,\omega)\rightarrow \sU_C(r, d-n_1(z), \omega')=\sU_{C,\,\omega'}$$ such that $\mathrm{H}_z^{\ast}\Theta_{\sU_{C,\,\omega'}}=\Theta_{\sU_{C,\,\omega}}$.
In fact, $\mathrm{H}_z$ is an isomorphism. For any parabolic bundle $E'$ with quasi-parabolic structure of type $\vec n'(z)$, let
$$F_{i}(E')_z=ker\{E'_z\twoheadrightarrow Q_i(E')_z\}\quad (1\le i\le l_z).$$
Then there exists a bundle $E$ and a homomorphism $E'\xrightarrow{i} E$ such that $F_{l_z}(E')_z=ker\{E'_z\xrightarrow{i_z}E\}.$
Let $F_1(E)_z=i_z(E'_z)\subset E_z$ and
$$F_{i+1}(E)_z=i_z(F_i(E')_z).$$
Then the quasi-parabolic structure of $E$ at $z\in I$ given by
$$0=F_{l_z+1}(E)_z\subset F_{l_z}(E)_z\subset F_{l_z-1}(E)_z \subset \cdots\subset F_1(E)_z\subset E_z$$
has of type $\vec n(z)=(n_1(z),...,n_{l_z+1}(z))$ and the weights $\vec a(z)$ are determined by
$\vec a'(z)$ (let $a_1(z)=0$, $a_2(z)=k-a'_{l_z+1}(z)$ and $a_{i+1}(z)=a'_i(z)+k-a'_{l_z+1}(z)$ for $2\le i\le l_z$).
The construction can be applied to a family of parabolic sheaves, which induces $H_z^{-1}$.
\end{proof}

\begin{lem}\label{lem4.10} For $\omega=(k,\{\vec n(x),\vec a(x)\}_{x\in I})$, if $n_1(z)>1$, let
\ga{4.7} {\omega''=(k, \{\vec n(x), \vec a(x)\}_{x\neq z \in I}\cup \{\vec a''(z), \vec n''(z)\})}
where $\vec a''(z)=(0,a_2(z),\cdots, a_{l_z+1}(z), k)$ (we assume $a_1(z)=0$) and
$$\vec n''(z)=(n_1(z)-m,n_2(z),\cdots, n_{l_z+1}(z), m), \quad 1<m<n_1(z).$$
Then $D_g(r, d-n_1(z),\omega')=D_g(r,d-m,\omega'').$
\end{lem}

\begin{proof}
 For a semistable parabolic bundle $E$ with parabolic structures determined by $\omega''$, let its quasi-parabolic structure at $z\in I$ is given by
 $$E_z\twoheadrightarrow Q_{l_z+1}(E)_z \twoheadrightarrow Q_{l_z}(E)_z\twoheadrightarrow \cdots \twoheadrightarrow Q_1(E)_z\twoheadrightarrow Q_0(E)_z=0.$$
Let $E'=ker\{E\twoheadrightarrow Q_1(E)_z\}$, then $E'$ has quasi-parabolic structure
$$E'_z \twoheadrightarrow Q'_{l_z}(E)_z\twoheadrightarrow \cdots \twoheadrightarrow Q'_1(E)_z\twoheadrightarrow Q'_0(E)_z=0$$
of type $\vec n'(z)=(n_2(z),...,n_{l_z+1}(z),n_1(z))$ at $z\in I$, where
$$Q'_i(E)_z=ker\{Q_{i+1}(E)_z\twoheadrightarrow Q_1(E)_z\}, \quad (1\leq i\leq l_z).$$
Then we show that $E'$ is a semistable parabolic bundle with parabolic structure determined by $\omega'$ if and only if
$E$ is a semistable parabolic bundle with parabolic structure determined by $\omega''$. In fact, by direct computation, we have
$$\chi(E)+\frac{1}{k}\sum^{l_z+2}_{i=1}a_i''(z)n_i''(z)=\frac{r}{k}a_2(z)+\chi(E')+\frac{1}{k}\sum^{l_z+1}_{i=1}a_i'(z)n_i'(z),$$
which implies that ${\rm p}ar_{\omega''}\mu(E)=\frac{a_2(z)}{k}+{\rm p}ar_{\omega'}\mu(E')$. For any sub-bundle $\sF\subset E$, let $\sF'\subset E'$ be the sub-bundle
such that $$0\to \sF'\to\sF\to Q_1(E)^{\sF}_z\to 0$$ is an exact sequence of sheaves. Then ${\rm p}ar_{\omega''}\mu(\sF)={\rm p}ar_{\omega'}\mu(\sF')+\frac{a_2(z)}{k}$.
Thus $E'$ is semistable if and only if $E$ is semistable. The construction can be applied to a family of parabolic sheaves, which induces
$$\sU_{C,\,\omega''}=\mathcal{U}_C(r,d-m,\omega'')\xrightarrow{\varphi} \sU_{C,\,\omega'}=\mathcal{U}_C(r,d-n_1(z),\omega').$$
One check directly that $\varphi^*\Theta_{\sU_{C,\,\omega'}}=\Theta_{\sU_{C,\,\omega''}}$ (i.e., it pulls back an ample line bundle to an
ample line bundle), which implies that $\varphi$ is a finite surjective morphism. To show that $\varphi$ is a injective morphism, which implies that
$\varphi$ is an isomorphism since $\sU_{C,\,\omega'}$ and $\sU_{C,\,\omega'}$ are normal projective varieties, we note
$Q'_i(E)_z=Q_{i+1}(E)^{E'}_z\subset Q_{i+1}(E)_z$ is the image of $E'_z\to E_z\twoheadrightarrow Q_{i+1}(E)_z$. Then $(E',Q'_{\bullet}(E)_z)=(E',Q_{\bullet+1}(E)^{E'}_z)$
is a parabolic subsheaf of $(E,Q_{\bullet+1}(E)_z)$ and we have exact sequence
$$0\to (E',Q_{\bullet+1}(E)^{E'}_z) \to (E,Q_{\bullet+1}(E)_z)\to ( \,_zQ_1(E)_z, Q_1(E)_{\bullet+1})\to 0$$
of parabolic sheaves, where
$$Q_1(E)_{\bullet+1}): Q_1(E)_z\twoheadrightarrow Q_1(E)_z\twoheadrightarrow\cdots\twoheadrightarrow Q_1(E)_z\twoheadrightarrow 0.$$
By direct computations, we have
$${\rm p}ar_{\omega''}\mu((E',Q_{\bullet+1}(E)^{E'}_z))={\rm p}ar_{\omega''}\mu((E,Q_{\bullet+1}(E)_z)).$$
Thus $(E,Q_{\bullet+1}(E)_z)$ is $s$-equivalent to $$(E',Q_{\bullet+1}(E)^{E'}_z)\oplus  ( \,_zQ_1(E)_z, Q_1(E)_{\bullet+1}),$$
which implies that $\varphi$ is a injective morphism, and we are done.
\end{proof}

\begin{rmks} (1) The moduli spaces $\sU_{C,\,\omega''}$ and theta line bundles $\Theta_{\sU_{C,\,\omega''}}$ are constructed
in \cite{Su3} for the case $a_{l_z+1}(z)-a_1(z)=k$, vanishing theorems can be generalized to this case.

(2) Let $\omega''=H_z^m(\omega)$, we will simply call $H_z^m(\omega)$ a Hecke transformation of $\omega$ at $z\in I$. Then
\ga{4.8} { D_g(r,d,\omega)=D_g(r, d-m, H^m_z(\omega)).}
\end{rmks}

Now we can prove another version of recurrence relation \eqref{4.4}, in which the degree $d$ is kept unchanged.

\begin{thm}\label{thm4.12} For any partitions $g=g_1+g_2$ and $I=I_1\cup I_2$, let
$$W_k=\{\,\lambda=(\lambda_1,...,\lambda_r)\,|\, 0=\lambda_r\le\lambda_{r-1}\le\cdots\le\lambda_1\le k\,\}$$
$$W'_k=\left\{\,\lambda\in W_k\,\,\mid\,\,\left(\sum_{x\in
I_1}\sum^{l_x}_{i=1}d_i(x)r_i(x)+\sum^r_{i=1}\lambda_i\right) \equiv 0({\rm mod}\,\,r)\right\}.$$
Then we have the following recurrence relation
\ga{4.9} {D_g(r,d,\omega)=\sum_{\mu\in W'_k} D_{g_1}(r, 0,\omega_1^{\mu})\cdot D_{g_2}(r,d,\omega_2^{\mu}).}
\end{thm}

\begin{proof} Let $P_k=\{\mu=(\mu_1,...,\mu_r)\,|\, 0\le\mu_r\le\cdots\le\mu_1< k\,\}$, by the recurrence relation \eqref{4.4}, we have
$$D_g(r,d,\omega)=\sum_{\mu\in Q_k}D_{g_1}(r,d_1^{\mu},\omega_1^{\mu})\cdot D_{g_2}(r,d_2^{\mu},\omega_2^{\mu})$$
where $Q_k=\{\mu=(\mu_1,\cdots, \mu_r)\in P_k\,|\, d_1^{\mu}\in \mathbb{Z}\}$. Recall definition of $d^{\mu}_j$ in
Theorem \ref{thm4.6}, which are integers such that $d_1^{\mu}+d_2^{\mu}=d$ and
\ga{4.10} {k(d_1^{\mu}+r)=k\cdot n_1^{\omega}+|\mu|,\quad |\mu|=\sum^r_{i=1}\mu_i}
where $n_1^{\omega}$ is the rational number defined in Theorem \ref{thm4.6}.

For $\mu=(\mu_1,\cdots, \mu_r),$ $0\leq \mu_r\leq \cdots \leq \mu_1\leq k$, let
$$H^1(\mu)=(k-\mu_{r-1}+\mu_r,\mu_1-\mu_{r-1},\mu_2-\mu_{r-1},\cdots, \mu_{r-2}-\mu_{r-1},0),$$
$H^{m}(\mu):=H^1(H^{m-1}(\mu))$ for $2\leq m\leq r$. Then, when $1\le m<r$,
$$H^m(\mu)_j=\left\{
\begin{array}{llll} k-\mu_{r-m}+\mu_{r-m+j} &\mbox{when $1\le j\le m$}\\
\mu_{j-m}-\mu_{r-m}&\mbox{when
$j>m $}\end{array}\right.$$
and $H^r(\mu)=(\mu_1-\mu_r,\mu_2-\mu_r,...,\mu_{r-1}-\mu_r,0)$. Moreover
\ga{4.11} { |H^m(\mu)|=\left\{
\begin{array}{llll} k\cdot m-r\cdot \mu_{r-m}+|\mu| &\mbox{when $m<r$}\\
-r\cdot\mu_r+|\mu|&\mbox{when
$m=r $}\end{array}\right. }
Let $0\le i^{\mu}<r$ be the unique integer such that $d_1^{\mu}\equiv i^{\mu}({\rm mod}\,\,r)$, let
$$\phi(\mu):=H^{r-i^{\mu}}(\mu).$$
Then, by \eqref{4.11}, it is easy to see that we have a map
\ga{4.12} {\phi:Q_k\to W'_k.}
One can check that $\omega_i^{\phi(\mu)}$ is a Hecke transformation of $\omega_i^{\mu}$ ($i=1,\,2$),
thus $D_{g_1}(r,d^{\mu}_1,\omega_1^{\mu})=D_{g_1}(r,0,\omega_1^{\phi(\mu)})$, $D_{g_2}(r,d_2^{\mu},\omega_2^{\mu})=D_{g_2}(r,d,\omega_2^{\phi(\mu)})$
by Lemma \ref{lem4.9} and Lemma \ref{lem4.10}. To prove the recurrence relation \eqref{4.9}, it is enough to show that
$\phi$ is bijective.

To prove the injectivity of $\phi$, let $\phi(\mu)=\phi(\mu')$, it is enough to show $i^{\mu}=i^{\mu'}$.
If both $i^{\mu}$ and $i^{\mu'}$ are nonzero, note $|\phi(\mu)|=k(r-i^{\mu})-r\mu_{i^{\mu}}+|\mu|$,
by $\phi(\mu)=\phi(\mu')$ and \eqref{4.10}, there exists a $q\in\mathbb{Z}$ such that
$$r\cdot(\mu'_{i^{\mu'}}-\mu_{i^{\mu}})=k(d_1^{\mu'}-i^{\mu'}-(d_1^{\mu}-i^{\mu}))=k\cdot r\cdot q.$$
Thus $k>|\mu'_{i^{\mu'}}-\mu_{i^{\mu}}|=k|q|$, which implies $q=0$ and $\mu'_{i^{\mu'}}=\mu_{i^{\mu}}$. If $i^{\mu}\neq i^{\mu'}$, let
$a=i^{\mu}-i^{\mu'}>0$,  then formula
\ga{4.13} {\phi(\mu)_j=\left\{
\begin{array}{llll} k-\mu_{i^{\mu}}+\mu_{j+i^{\mu}} &\mbox{when $1\le j\le  r-i^{\mu}$}\\
\mu_{j-r+i^{\mu}}-\mu_{i^{\mu}}&\mbox{when
$j>r-i^{\mu} $}\end{array}\right.}
implies $\mu_a=k+\mu_r'\ge k$, which is a contradiction since $\mu\in Q_k$. If
$i^{\mu}=0$, $i^{\mu'}$ must be zero. Otherwise, the same arguments imply $\mu'_{i^{\mu'}}=\mu_r$ and $\mu_j=k+\mu'_{i^{\mu'}+j}$
for all $1\le j\le r-i^{\mu'}$.

To prove that $\phi$ is surjective, by using \eqref{4.11}, \eqref{4.10} becomes
\ga{4.14} {\frac{k\cdot n_1^{\omega}+|\phi(\mu)|}{r}=\left\{
\begin{array}{llll} k\cdot\frac{d_1^{\mu}+2r-i^{\mu}}{r}-\mu_{i^{\mu}} &\mbox{when $i^{\mu}>0$}\\
k\cdot\frac{d_1^{\mu}+r}{r}-\mu_r &\mbox{when $i^{\mu}=0 $}\end{array}\right.}
For any $\lambda=(\lambda_1,...,\lambda_{r-1},0)\in W'_k$, there are unique integers $q^{\lambda}$ and $0\le r^{\lambda}<k$ such that
$$\frac{k\cdot n_1^{\omega}+|\lambda|}{r}=k\cdot q^{\lambda}-r^{\lambda}.$$
If $\lambda_1+r^{\lambda}<k$, let $\mu=(\lambda_1+r^{\lambda}, ..., \lambda_{r-1}+r^{\lambda},r^{\lambda})\in P_k$, then $d_1^{\mu}=r(q^{\lambda}-1)$
by \eqref{4.10}. Thus $i^{\mu}=0$ and $\phi(\mu)=\lambda$. If $\lambda_1+r^{\lambda}\ge k$, since $\lambda_r+r^{\lambda}<k$, there exists an unique $1\le i_0\le r-1$ such that $$\lambda_{i_0}+r^{\lambda}\ge k, \qquad \lambda_{i_0+1}+r^{\lambda}< k.$$
Let $\mu_j=\lambda_{i_0+j}+r^{\lambda}$ ($1\le j\le r-i_0$) and $\mu_{r-i_0+j}=\lambda_j+r^{\lambda}-k$ ($1\le j\le i_0$).
Then $\mu=(\mu_1,...,\mu_r)\in Q_k$ with $d_1^{\mu}=r(q^{\lambda}-1)-i_0$ and $i^{\mu}=r-i_0$. It is easy to see that $\phi(\mu)=\lambda$.

\end{proof}

\section {A finite dimensional proof of Verlinde formula}

As an application of the recurrence relation \eqref{4.3} and \eqref{4.9}, we prove a closed formula of $D_g(r,d,\omega)$ (the so called Verlinde formula).
Recall
\ga{5.1} {S_{\lambda}(z_1,...,z_r)=\frac{|z_j^{\lambda_i+r-i}|}{|z_j^{r-i}|}=\frac{|z_j^{\lambda_i+r-i}|}{\Delta(z_1,...,z_r)}} is the so called
Schur polynomial of $\lambda=(\lambda_1\ge\lambda_2\ge\cdots\ge\lambda_r\ge 0)$,
$$\Delta(z_1,...,z_r)=\prod_{i<j}(z_i-z_j).$$
We give here an detail proof of some identities of Schur polynomials.

\begin{prop}\label{prop5.1} For $\vec v=(v_1,\ldots,v_r)$, $0\le v_r<\cdots <v_1<r+k$, let
$$S_{\lambda}\left({\rm exp}\,2\pi i\frac{\vec v}{r+k}\right)=S_{\lambda}(e^{2\pi i\frac{v_1}{r+k}},...,e^{2\pi i\frac{v_r}{r+k}}),$$
$P_k=\{\mu=(\mu_1,...,\mu_r)\,|\, 0\le\mu_r\le\cdots\le\mu_1< k\,\}$, $|\mu|:=\sum\mu_i$.
Then
\ga{5.2} {\aligned&\sum_{\mu\in P_k}S_{\mu}\left({\rm exp}\,2\pi i\frac{\vec v}{r+k}\right)\cdot S_{\mu^*}\left({\rm exp}\,2\pi i\frac{\vec v}{r+k}\right)\\&=
{\rm exp}\left(2\pi i\frac{k}{r+k}|\vec v|\right)\cdot\frac{k(r+k)^{r-1}}{\prod_{i<j}\left(2\sin\,\pi \frac{v_i-v_j}{r+k}\right)^2},\endaligned}
let $W_k=\{\mu=(\mu_1,...,\mu_r)\,|\,0=\mu_r\le\mu_{r-1}\le\cdots<\mu_1\le k\,\}$, we have
\ga{5.3}{\aligned&\sum_{\mu\in W_k} S_{\mu}\left({\rm exp}\,2\pi i\frac{\vec v}{r+k}\right)\cdot S_{\mu^*}\left({\rm exp}\,2\pi i\frac{\vec {v}}{r+k}\right)\\&={\rm exp}\left(2\pi i\frac{k}{r+k}|\vec v|\right)\cdot\frac{r(r+k)^{r-1}}{\prod_{i<j}\left(2\sin\,\pi \frac{v_i-v_j}{r+k}\right)^2}\endaligned}
and, if $\vec v\neq\vec{v'}$,
\ga{5.4}{\aligned \sum_{\mu\in W_k}
&{\rm exp}\,2\pi i\frac{-|\mu|\cdot |\vec v|}{r(r+k)}\cdot
{\rm exp}\,2\pi i\frac{-|\mu^*|\cdot |\vec {v'}|}{r(r+k)}\cdot\\& S_{\mu}\left({\rm exp}\,2\pi i\frac{\vec v}{r+k}\right)\cdot S_{\mu^*}\left({\rm exp}\,2\pi i\frac{\vec {v'}}{r+k}\right)
=0.\endaligned}

\end{prop}

\begin{proof} To prove \eqref{5.2}, since $\mathbb{S}_{\mu^*}(V)={\rm det}(V)^k\otimes \mathbb{S}_{\mu}(V^*)$, we have
$$S_{\mu^*}\left({\rm exp}\,2\pi i\frac{\vec v}{r+k}\right)=\overline{S_{\mu}\left({\rm exp}\,2\pi i\frac{\vec v}{r+k}\right)}{\rm exp}\left(2\pi i\frac{k}{r+k}\sum_{i=1}^{r}v_i\right).$$ Thus it is enough to show that
\ga{5.5}{\aligned&\sum_{\mu\in P_k}S_{\mu}\left({\rm exp}\,2\pi i\frac{\vec v}{r+k}\right)\cdot \overline{S_{\mu}\left({\rm exp}\,2\pi i\frac{\vec v}{r+k}\right)}\\&=
\frac{k(r+k)^{r-1}}{\prod_{i<j}\left(2\sin\,\pi \frac{v_i-v_j}{r+k}\right)^2}.\endaligned}

For $\lambda=(\lambda_1,...,\lambda_r)$, the functions $e^{\tau(\lambda)}$, $J(e^{\lambda})$ are defined by
$$\aligned &e^{\tau(\lambda)}(\mathrm{diag}(t_1, \cdots,t_r)):=t_1^{\lambda_{\tau(1)}}\cdot \cdots\cdot t_r^{\lambda_{\tau(r)}}\\& J(e^{\lambda})(\mathrm{diag}(t_1, \cdots,t_r)):=\sum_{\tau\in \mathfrak{S}_r}\epsilon(\tau)e^{\tau(\lambda)}(\mathrm{diag}(t_1, \cdots,t_r)),\endaligned$$
where $\tau(\lambda)=(\lambda_{\tau(1)},...,\lambda_{\tau(r)})$, $\mathfrak{S}_r$ is the symmetric group. Let
$$\Delta(\vec v)=\prod_{i<j}(e^{2\pi i\frac{v_i}{r+k}}-e^{2\pi i\frac{v_j}{r+k}})$$
and $\rho=(r-1, r-2, ..., 0)$. By expansion of determinant, we have
$$\aligned &S_{\mu}\left({\rm exp}\,2\pi i\frac{\vec v}{r+k}\right)=\frac{1}{\Delta(\vec v)}\sum_{\tau\in \mathfrak{S}_r}\epsilon(\tau)e^{2\pi i\frac{\mu_1+r-1}{r+k}v_{\tau(1)}}
\cdot\cdots\cdot e^{2\pi i\frac{\mu_r}{r+k}v_{\tau(r)}}\\=&\frac{1}{\Delta(\vec v)}\sum_{\tau\in \mathfrak{S}_r}\epsilon(\tau)e^{\tau(\vec v)}\left({\rm exp}\,2\pi i\frac{\mu+\rho}{r+k}\right)=\frac{1}{\Delta(\vec v)}J(e^{\vec v})(t_{\mu}),\endaligned$$
where $t_{\mu}={\rm exp}\,2\pi i\frac{\mu+\rho }{r+k}$.
By $\Delta(\vec v)\overline{\Delta(\vec v)}=\prod_{i<j}\left(2\sin\,\pi \frac{v_i-v_j}{r+k}\right)^2$, we have
\ga{5.6} {\aligned &S_{\mu}\left({\rm exp}\,2\pi i\frac{\vec v}{r+k}\right)\cdot \overline{S_{\mu}\left({\rm exp}\,2\pi i\frac{\vec v}{r+k}\right)}
=\\&\frac{1}{\prod_{i<j}\left(2\sin\,\pi \frac{v_i-v_j}{r+k}\right)^2}J(e^{\vec v)})(t_{\mu})
\cdot \overline{J(e^{\vec v})(t_{\mu})}.\endaligned}

Let $T_k=\{\,t=\mathrm{diag}(e^{\frac{2\pi i}{r+k}t_1}, \cdots, e^{\frac{2\pi i}{r+k}t_r})\,|\, 0\le t_i<r+k\,\}\subset {\rm GL}(r)$ be the subgroup and
$T_k^{reg}=\{\,t\in T_k\,|\, t_i\ne t_j \text{if $i\neq j$}\,\}$. The group $\mathfrak{S}_r$ acts on $T_k$ by
$\tau(t)=\mathrm{diag}(e^{\frac{2\pi i}{r+k}t_{\tau(1)}}, \cdots, e^{\frac{2\pi i}{r+k}t_{\tau(r)}})$ and the functions
$$\aligned J(e^{\lambda})(t)&=\sum_{\tau\in \mathfrak{S}_r}\epsilon(\tau)e^{2\pi i\frac{\lambda_{\tau(1)}}{r+k}t_{1}}
\cdot\cdots\cdot e^{2\pi i\frac{\lambda_{\tau(r)}}{r+k}t_{r}}\\&=
\sum_{\tau\in \mathfrak{S}_r}\epsilon(\tau)e^{2\pi i\frac{\lambda_1}{r+k}t_{\tau(1)}}
\cdot\cdots\cdot e^{2\pi i\frac{\lambda_r}{r+k}t_{\tau(r)}}\endaligned$$
for any $\lambda=(\lambda_1,...,\lambda_r)$ are ant-symmetric functions, thus $J(e^{\lambda})(t)=0$ if $t\notin T_k^{reg}$. It is clear that $\mathfrak{S}_r$ acts
on $T_k^{reg}$ freely and
$$T^{reg}_k=\bigcup_{\mu\in \bar P_k}\mathfrak{S}_r\cdot t_{\mu},\quad t_{\mu}={\rm exp}\,2\pi i\frac{\mu+\rho }{r+k}.$$
The right hand side of \eqref{5.6} is a symmetric function, we have
$$\aligned &\sum_{\mu\in \bar P_k}S_{\mu}\left({\rm exp}\,2\pi i\frac{\vec v}{r+k}\right)\cdot \overline{S_{\mu}\left({\rm exp}\,2\pi i\frac{\vec v}{r+k}\right)}=
\frac{1}{|\mathfrak{S}_r|}\\&
\prod_{i<j}\left(2\sin\,\pi \frac{v_i-v_j}{r+k}\right)^{-2}\sum_{t\in T_k}J(e^{\vec v})(t)
\overline{J(e^{\vec v})(t)}\endaligned$$ where $\bar P_k=\{\mu=(\mu_1,...,\mu_r)\,|\,0\le\mu_r\le\cdots\le\mu_1\le k\}$.
To compute
$$\aligned&\sum_{t\in T_k}J(e^{\vec v})(t)
\overline{J(e^{\vec v})(t)}=\sum_{\tau,\sigma\in\mathfrak{S}_r}\epsilon(\tau)\cdot \epsilon(\sigma)\sum_{t\in T_k}
e^{\tau(\vec v)}(t)\cdot\overline{e^{\sigma(\vec v)}(t)},\endaligned$$
note $e^{\tau(\vec v)}$ and $e^{\sigma(\vec v)}$ are different character of $T_k$ when $\tau\neq \sigma$, we have
$$\sum_{t\in T_k}J(e^{\vec v})(t)\overline{J(e^{\vec v})(t)}=|\mathfrak{S}_r|\cdot |T_k|.$$  Thus
\ga{5.7}{\aligned&\sum_{\mu\in \bar P_k}S_{\mu}\left({\rm exp}\,2\pi i\frac{\vec v}{r+k}\right)\cdot \overline{S_{\mu}\left({\rm exp}\,2\pi i\frac{\vec v}{r+k}\right)}
\\&=\frac{(r+k)^r}{\prod_{i<j}\left(2\sin\,\pi \frac{v_i-v_j}{r+k}\right)^2}.\endaligned}

For $\mu\in P'_k:=\bar P_k\setminus P_k$, let
$t_{\mu}'=\mathrm{diag}(1, e^{2\pi i\frac{\mu_2+r-1}{r+k}}, \cdots, e^{2\pi i\frac{\mu_r+1}{r+k}} )$ and
$$T'_k=\{\,t=\mathrm{diag}(1, e^{\frac{2\pi i}{r+k}t_2}, \cdots, e^{\frac{2\pi i}{r+k}t_r})\,|\, 0\le t_i<r+k\,\}\subset T_k$$ be the subgroup and
${T'}_k^{reg}=T'_k\cap T^{reg}_k$. Then
$${T'}^{reg}_k=\bigcup_{\mu\in  P'_k}\mathfrak{S}_{r-1}\cdot t'_{\mu}.$$
Note $J(e^{\vec v})(t_{\mu})=e^{-2\pi i\frac{|\vec v|}{r+k}}J(e^{\vec v})(t'_{\mu})$, $J(e^{\vec v})(t)=0$ if $t\notin {T'}_k^{reg}$, we have
\ga{5.8} {\aligned &\sum_{\mu\in P'_k}S_{\mu}\left({\rm exp}\,2\pi i\frac{\vec v}{r+k}\right)\cdot \overline{S_{\mu}\left({\rm exp}\,2\pi i\frac{\vec v}{r+k}\right)}=\\&
\frac{1}{|\mathfrak{S}_{r-1}|}
\prod_{i<j}\left(2\sin\,\pi \frac{v_i-v_j}{r+k}\right)^{-2}\sum_{t\in  T'_k}J(e^{\vec v})(t)
\overline{J(e^{\vec v})(t)}\\&=\frac{r(r+k)^{r-1}}{\prod_{i<j}\left(2\sin\,\pi \frac{v_i-v_j}{r+k}\right)^2}.\endaligned}
Thus \eqref{5.7} and \eqref{5.8} imply the formula \eqref{5.2}. The proof of formula \eqref{5.3} is similar with formula \eqref{5.8}, we omit it.

Now we are going to prove formula \eqref{5.4}. To simpify notation, let
$$G_{\mu}(\vec v):={\rm exp}\,2\pi i\frac{-|\mu|\cdot |\vec {v}|}{r(r+k)}\cdot S_{\mu}\left({\rm exp}\,2\pi i\frac{\vec v}{r+k}\right).$$
Then it is equivalent to prove that, when $\vec v\neq\vec{v'}$, we have
\ga{5.9}{\sum_{\mu\in W_k}G_{\mu}(\vec v)\overline{G_{\mu}(\vec {v'})}=0.}
Let $\lambda^{\mu}=(\lambda_1^{\mu},...,\lambda^{\mu}_r)$ with $\lambda_i^{\mu}=\mu_i+r-i-\frac{|\mu|+|\rho|}{r}$, then
$$\aligned G_{\mu}(\vec v)&=\frac{{\rm exp}\,2\pi i\frac{|\rho|\cdot |\vec {v}|}{r(r+k)}}{\Delta(\vec v)}\sum_{\tau\in \mathfrak{S}_r}\epsilon(\tau)e^{2\pi i\frac{\lambda^{\mu}_1}{r+k}v_{\tau(1)}}
\cdot\cdots\cdot e^{2\pi i\frac{\lambda^{\mu}_r}{r+k}v_{\tau(r)}}\\&=\frac{{\rm exp}\,2\pi i\frac{|\rho|\cdot |\vec {v}|}{r(r+k)}}{\Delta(\vec v)}\sum_{\tau\in \mathfrak{S}_r}\epsilon(\tau)e^{\tau(\vec v)}\left({\rm exp}\,2\pi i\frac{\lambda^{\mu}}{r+k}\right)\\&=\frac{{\rm exp}\,2\pi i\frac{|\rho|\cdot |\vec {v}|}{r(r+k)}}{\Delta(\vec v)}J(e^{\vec v})(t_{\lambda^{\mu}}),\quad t_{\lambda^{\mu}}={\rm exp}\,2\pi i\frac{\lambda^{\mu}}{r+k}.\endaligned$$
Since $e^{\sigma(\vec v)}$, $e^{\tau(\vec {v'})}$ ($\forall\,\sigma,\,\tau\in\mathfrak{S}_r$) are different
characters of a subgroup $T_k=\{\,t=\mathrm{diag}(e^{\frac{2\pi i}{r+k}t_1}, \cdots, e^{\frac{2\pi i}{r+k}t_r})\,|\sum t_i=0,\,t_i-t_j\in\mathbb{Z}\,\}\subset {\rm GL}(r)$
whenever $\vec v\neq \vec {v'}$, we have
$$\aligned \sum_{\mu\in W_k}G_{\mu}(\vec v)\overline{G_{\mu}(\vec {v'})}&=
\frac{{\rm exp}\,2\pi i\frac{|\rho|\cdot (|\vec {v}|-|\vec{v'}|)}{r(r+k)}}{\Delta(\vec v)\overline{\Delta(\vec v')}}
\sum_{\mu\in W_k}J(e^{\vec v})(t_{\lambda^{\mu}})\cdot\overline{J(e^{\vec {v'}})(t_{\lambda^{\mu}})}\\&=
\frac{{\rm exp}\,2\pi i\frac{|\rho|\cdot (|\vec {v}|-|\vec{v'}|)}{r(r+k)}}{\Delta(\vec v)\overline{\Delta(\vec v')}|\mathfrak{S}_r|}
\sum_{t\in T_k}J(e^{\vec v})(t)\cdot\overline{J(e^{\vec {v'}})(t)}=0.
\endaligned$$
\end{proof}

\begin{nota}\label{nota5.1} For $\vec n(x)=(n_1(x),n_2(x),\cdots,n_{l_x+1}(x))$ and $$\vec a(x)=(a_1(x),a_2(x),\cdots,a_{l_x+1}(x))$$ with
$\sum n_i(x)=r$, $0\leq a_1(x)<a_2(x)<\cdots
<a_{l_x+1}(x)< k$, define
\ga{10} {\lambda_x=(\,\,\overbrace{\lambda_1,\ldots,\lambda_1}^{n_1(x)}\,,\,\,\overbrace{\lambda_2,\ldots,\lambda_2}^{n_2(x)}\,,\,\,\ldots,\, \,\overbrace{\lambda_{l_x+1},\ldots,\lambda_{l_x+1}}^{n_{l_x+1}(x)}\,\,)}
where $\lambda_i=k-a_i(x)$ ($1\le i\le l_x+1$).
\end{nota}

\begin{thm}\label{thm5.3} For given data $\omega=(k,\{\vec n(x),\vec a(x)\}_{x\in I})$, let
$$S_{\omega}(z_1,...,z_r)=\prod_{x\in I}S_{\lambda_x}(z_1,...,z_r),\quad |\omega|=\sum_{x\in I}|\lambda_x|$$
where $S_{\lambda_x}(z_1,...,z_r)$ are Schur polynomials and $|\lambda_x|$ denotes the total number of boxes in a Young diagram
associated to $\lambda_x$. Then
\ga{5.11}{\aligned &D_g(r,d,\omega)=(-1)^{d(r-1)}\left(\frac{k}{r}\right)^g(r(r+k)^{r-1})^{g-1}\\&\sum_{\vec v}
\frac{{\rm exp}\left(2\pi i\left(\frac{d}{r}-\frac{|\omega|}{r(r+k)}\right)\sum_{i=1}^{r}v_i\right)S_{\omega}\left({\rm exp}\,2\pi i\frac{\vec v}{r+k}\right)} {\prod_{i<j}\left(2\sin\,\pi \frac{v_i-v_j}{r+k}\right)^{2(g-1)}}\endaligned}
where $\vec v=(v_1,v_2,\ldots,v_r)$ runs through the integers $$0=v_r<\cdots <v_2<v_1<r+k.$$
\end{thm}

\begin{proof} It is easy to check the formula when $g=0$ and $|I|\le 2$.
The case $g=0$, $d=0$ and $|I|=3$ is much more involved (we leave for another occasion).
Then the proof is done by the following lemmas.
\end{proof}

\begin{lem}\label{lem5.4} If the formula \eqref{5.11} holds when $g=0$, then it holds for any $g>0$.
\end{lem}

\begin{proof} It is easy to see that both side of \eqref{5.11} are invariant when any $\lambda_x$ is replaced by $\lambda_x'=\lambda_x+(a,...,a)$.
By the recurrence relation \eqref{4.3},
\ga{5.12} {D_g(r,d,\omega)=\sum_{\mu}D_{g-1}(r,d,\omega^{\mu})}
where $\omega^{\mu}=(k, \{\vec n(x),\,\vec a(x)\}_{x\in I\cup\{x_1,\,x_2\}})$ was defined in Notation \ref{nota4.5} and
$\mu=(\mu_1,\ldots,\mu_r)$ runs through the integers $0\le\mu_r\le\cdots\le\mu_1< k$.

It is easy to check that $\lambda_{x_2}=(k-\mu_r,\ldots,k-\mu_1)$ and
$$\lambda_{x_1}=(\mu_1,\ldots,\mu_r)+(\mu_1+\mu_r-k,\,\mu_1+\mu_r-k,\ldots,\, \mu_1+\mu_r-k)$$
(in Notation \ref{nota5.1}). Thus, without loss of generality, we can assume
$$\lambda_{x_1}=\mu=(\mu_1,\ldots,\mu_r),\quad \lambda_{x_2}=\mu^*=(k-\mu_r,\ldots,k-\mu_1).$$
Assume that formula \eqref{5.11} holds for $g-1$, then
\ga{5.13}{\aligned &D_{g-1}(r,d,\omega^{\mu})=(-1)^{d(r-1)}\left(\frac{k}{r}\right)^{g-1}(r(r+k)^{r-1})^{g-2}\\&\sum_{\vec v}
\frac{{\rm exp}\left(2\pi i\left(\frac{d}{r}-\frac{|\omega^{\mu}|}{r(r+k)}\right)\sum_{i=1}^{r}v_i\right)S_{\omega^{\mu}}\left({\rm exp}\,2\pi i\frac{\vec v}{r+k}\right)} {\prod_{i<j}\left(2\sin\,\pi \frac{v_i-v_j}{r+k}\right)^{2(g-2)}}\endaligned}
where $|\omega^{\mu}|=|\omega|+k\cdot r$, $S_{\omega^{\mu}}
=S_{\omega}\cdot S_{\mu}\cdot S_{\mu^*}$. By \eqref{5.12} and \eqref{5.13},
$$\aligned&D_g(r,d,\omega)=(-1)^{d(r-1)}\left(\frac{k}{r}\right)^g(r(r+k)^{r-1})^{g-1}\\&\sum_{\vec v}
\frac{{\rm exp}\left(2\pi i\left(\frac{d}{r}-\frac{|\omega|}{r(r+k)}\right)\sum_{i=1}^{r}v_i\right)S_{\omega}\left({\rm exp}\,2\pi i\frac{\vec v}{r+k}\right)} {\prod_{i<j}\left(2\sin\,\pi \frac{v_i-v_j}{r+k}\right)^{2(g-1)}}\\&{\rm exp}\left(-2\pi i\frac{k}{r+k}\sum_{i=1}^{r}v_i\right)
\frac{\prod_{i<j}\left(2\sin\,\pi \frac{v_i-v_j}{r+k}\right)^2}{k(r+k)^{r-1}}\\&
\sum_{\mu}S_{\mu}\left({\rm exp}\,2\pi i\frac{\vec v}{r+k}\right)\cdot S_{\mu^*}\left({\rm exp}\,2\pi i\frac{\vec v}{r+k}\right).\endaligned$$
Then the formula \eqref{5.11} holds by the identity \eqref{5.2} in Proposition \ref{prop5.1}.

\end{proof}

\begin{lem}\label{lem5.5} If the formula \eqref{5.11} for $D_0(r,d,\omega)$ holds when $|I|\le3$, then it holds for all $D_0(r,d,\omega)$.
\end{lem}

\begin{proof} The proof is by induction on the number of parabolic points. Let $V_g(r,d,\omega)$ denote the right hand side of
formula \eqref{5.11} (the Verlinde number). By Theorem \ref{thm4.12}, let $I=I_1\cup I_2$ with $|I_1|=2$, we have
$$D_0(r,d,\omega)=\sum_{\mu\in W'_k}V_0(r,0,\omega_1^{\mu})\cdot V_0(r,d,\omega_2^{\mu}).$$
It is not difficult to check that $V_0(r,0,\omega_1^{\mu})=0$ for $\mu\in W_k\setminus W'_k$. Thus
$$\aligned &D_0(r,d,\omega)=\sum_{\mu\in W_k}V_0(r,0,\omega_1^{\mu})\cdot V_0(r,d,\omega_2^{\mu})=\frac{(-1)^{d(r-1)}}{(r(r+k)^{r-1})^2}\sum_{\vec v,\,\vec {v'}}\\&
\frac{{\rm exp}\left(2\pi i\left(-\frac{|\omega_1|}{r(r+k)}\right)|\vec v|\right)} {\prod_{i<j}\left(2\sin\,\pi \frac{v_i-v_j}{r+k}\right)^{-2}}
\cdot\frac{{\rm exp}\left(2\pi i\left(\frac{d}{r}-\frac{|\omega_2|}{r(r+k)}\right)|\vec {v'}|\right)} {\prod_{i<j}\left(2\sin\,\pi \frac{v'_i-v'_j}{r+k}\right)^{-2}}\cdot\\&
S_{\omega_1}\left({\rm exp}\,2\pi i\frac{\vec v}{r+k}\right)\cdot S_{\omega_2}\left({\rm exp}\,2\pi i\frac{\vec {v'}}{r+k}\right)\cdot\sum_{\mu\in W_k}
{\rm exp}\,2\pi i\frac{-|\mu|\cdot |\vec v|}{r(r+k)}\\&
{\rm exp}\,2\pi i\frac{-|\mu^*|\cdot |\vec {v'}|}{r(r+k)}\cdot S_{\mu}\left({\rm exp}\,2\pi i\frac{\vec v}{r+k}\right)\cdot S_{\mu^*}\left({\rm exp}\,2\pi i\frac{\vec {v'}}{r+k}\right)
\endaligned$$ and we are done by \eqref{5.3} and \eqref{5.4} of Proposition \ref{prop5.1}.
\end{proof}

\bibliographystyle{plain}

\renewcommand\refname{References}

\end{document}